\documentclass{gAPA2e}

\usepackage{subfigure}

\theoremstyle{plain}
\newtheorem{theorem}{Theorem}[section]
\newtheorem{corollary}[theorem]{Corollary}

\theoremstyle{remark}

\theoremstyle{definition}

\usepackage{color}
\usepackage{mathtools}

\usepackage{float}

\usepackage{autonum}

\newcommand{{\cX}}{\mathcal X}
\newcommand{\cY}{\mathcal Y}

\newcommand{\dd}{{\mathrm d}}

\newcommand{\bu}{\overline{u}}

\newcommand{\Yl}{\cY^{\|}}
\newcommand{\Yp}{\cY^{\perp}}

\newcommand{\rp}{r_{\perp}}

\newcommand{\dll}{\vartheta^{(j)}_{\|}}
\newcommand{\dpp}{\vartheta^{(j)}_{\perp}}

\newcommand{\beq}{\begin{equation}}
\newcommand{\eeq}{\end{equation}}
\newcommand{\bea}{\begin{eqnarray}}
\newcommand{\eea}{\end{eqnarray}}
\newcommand{\beas}{\begin{eqnarray*}}
\newcommand{\eeas}{\end{eqnarray*}}

\def\b1{{\bf 1}}

\def\cO{{\mathcal O}}

\newcommand{\rd}{\mathrm{d}}

\newcommand{\cls}[1]{{\color{black}{#1}}}

\newcommand{\id}{\operatornamewithlimits{id}}

\usepackage{braket}

\begin{document}



\title{{\itshape Convergence Analysis of Ensemble Kalman Inversion:\\ The Linear, Noisy Case}}

\author{C. Schillings$^{\rm a}$$^{\ast}$\thanks{$^\ast$Corresponding author. Email: c.schillings@uni-mannheim.de
\vspace{6pt}} and A.M. Stuart$^{\rm b}$\\\vspace{6pt}  $^{a}${\em{Institute for Mathematics, 
University of Mannheim,
A5, 6, 68131 Mannheim, 
Germany}};
$^{b}${\em{Department of Computing and Mathematical Sciences, California Institute of Technology, CA 91125, USA}}\\ }

\maketitle

\begin{abstract}
We present an analysis of ensemble Kalman inversion, based on the continuous time limit of the algorithm. The analysis of the dynamical behaviour of the ensemble allows \cls{us}
to establish well-posedness and convergence results for a fixed ensemble size. We will build on the results presented in \cite{AndrewC2017} and generalise them to the case of noisy observational data, in particular the influence of the noise on the convergence will be investigated, both theoretically and numerically. 
\cls{We focus on linear inverse problems where a very complete theoretical
analysis is possible.}

\begin{classcode} 65N21, 62F15, 65N75 \end{classcode}

\end{abstract}

\section{Introduction}

\cls{The Kalman filter has been enormously successful since its
introduction in the 1960s as a state estimation tool for
linear Gaussian systems in both discrete or continuous time;
see \cite[sections 4.1 and 8.1]{KLS} and the references therein. 
A natural generalisation
to nonlinear state estimation is the extended Kalman filter
\cite[sections 4.2.2 and 8.2.2]{KLS} and this was proposed
as a method for numerical weather prediction in \cite{Ghil81}.
The ensemble Kalman filter  \cite{Evensen06} 
was introduced in state estimation problems
as a way of circumventing the need to compute enormous covariance matrices
when applying the extended Kalman filter to large problems such
as those arising in atmosphere or ocean dynamics \cite{evensen2003ensemble,evensen1996assimilation,houtekamer2001sequential}. 
The inherent parallelisability of the method, together with its effectiveness
in state estimation, has made it very popular and its use spread outside
the atmosphere-ocean sciences community. In particular it has been
widely adopted by the oil industry for subsurface 
inversion \cite{oliver2008inverse}. Building on this applied work in subsurface
inversion, in \cite{iglesias2013ensemble} a generic ensemble Kalman inversion 
tool for inverse problems in the form
$$y={\mathcal G}(u)+\eta$$
was formulated;  here the objective is to recover $u$ from $y$, a noisy observation of
${\mathcal G}(u)$ and $\eta$ denotes the noise. 
Despite documented success as a solver for such inverse problems,
there is very little analysis of the algorithm. Essentially
two facts are known about the finite ensemble size regime in which it is
used: that the basic form of the iteration preserves
the linear span of the initial 
ensemble \cite{li2007iterative,iglesias2013ensemble}; and that for the linear
noise free problem the method is a discretisation of a set of interacting
gradient flows for the output least squares objective function associated
with the linear inverse problem \cite{AndrewC2017}. The combination 
of these two facts allows an almost complete analysis of the algorithm
in the setting of the linear noise free inverse problem. The purpose
of this paper is to extend those results to include the effect of noise.

It is of interest to give some insight into where the gradient flow
structure comes from in this problem. Inspection of the Kalman-Bucy
filter \cite[section 8.1]{KLS} reveals that when the drift of the signal is
zero and the observed data is constant, then the equation for the
mean is a gradient flow for the output least squares function related
to the observation operator, preconditioned by the  covariance.
In \cite{bergemann2010localization} this observation was used to
create algorithms for the analysis step in state estimation problems
employing the Kalman filter, essentially by replacing the Kalman
variance by the empirical covariance; a resulting gradient structure
was noted and exploited. The Kalman-Bucy filter
with no drift in the signal, and the analysis phase of the
general filter with linear observations, are closely related to solution of a linear inverse problem. As a consequence it not unnatural that
in \cite{AndrewC2017} it was demonstrated that the continuous time limit
of the ensemble Kalman inversion algorithm is an interacting set of
gradient flows.

There are two ways of viewing algorithms for ensemble Kalman inversion. The
first is simply as derivative free optimisers, in which the ensemble is used as a proxy
for derivative information; this is the view put forward in
\cite{iglesias2013ensemble}. The second is as a method to solve
a Bayesian inverse problem. We adopt the first viewpoint throughout the
paper, essentially because, as the literature survey in the next
paragraph explains,
there is little hope of rigorous uncertainty quantification via
ensemble methods, except for linear problems. And, although our analysis
is limited to the linear problem, our goal is to obtain insight into
ensemble inversion methods in general.

The Bayesian approach to distributed parameter
inversion allows incorporation of both model
and data uncertainties and leads to a complete characterisation of 
the uncertainty via the posterior distribution; 
see \cite{Stuart_2010,DashtiStuart}. However, for computationally 
intensive applications, the computation or approximation of the posterior 
is, even with today's supercomputers, often intractable. Thus
ensemble inversion provides an attractive alternative which, through the
ensemble, may include some information about uncertainties. 
The low computational costs, the straightforward implementation and its non-intrusive nature  make the method appealing.  In the state estimation 
context \cite{reich2015probabilistic}, 
well-posedness results for the EnKF can be found in \cite{kelly2014well,tongnonlinear,tonginfl,Kelly25082015} and a large-time convergence analysis in the case of a fully observed system is presented in \cite{2016arXiv161206065D}; other
interesting methods and analyses may be found in \cite{bergemann2010localization,bergemann2010mollified, reich2011}. The analysis of the large ensemble size limit can be found in \cite{kwiatkowski2015convergence, gratton2014convergence}. For inverse 
problems, the large ensemble size limit is studied in \cite{ernst2015analysis}
and, importantly for the optimisation perspective we take in this paper, demonstrated
to differ from the true posterior distribution except in the linear case.
In ensemble inversion,
the connection to deterministic regularisation techniques and step-size strategies for nonlinear forward problems is developed 
in \cite{iglesias2013ensemble, iglesias2014iterative, iglesias2015regularizing}. }

The linear inverse problem which we study in this paper is defined as follows:
let $\cls{\cX}$ \cls{denote a separable Hilbert space}. Furthermore, we denote by \cls{$A\in\mathcal L(\cls{\cX},\mathbb R^K)$} the forward response operator mapping from the parameter space $\cls{\cX}$ to the data space \cls{$\mathbb R^K$}. The observations are assumed to be finite-dimensional, i.e. \cls{the forward response operator maps to $\mathbb R^K$, where $K\in\mathbb N$ denotes the number of observations}. The goal of computation is to recover the unknown parameters $u$ from noisy observations $y$, where
\begin{equation}
\label{eq:ip}
y=Au+\eta.
\end{equation}
The noise $\eta$ in the observations is assumed to be normally distributed with $\eta\sim\mathcal N(0,\Gamma)$, $\Gamma\in\mathbb R^{K\times K}$ symmetric positive definite. \cls{In the Bayesian setting, the unknown parameter $u$ is 
interpreted as a random variable or random field, distributed according to prior 
$\mu_0$. The Bayesian solution to the inverse problem is the conditional distribution
of $u$ given $y$, and to define this it is necessary to make an assumption
on the a priori dependence structure between $u$ and $\eta;$ 
it is often assumed that the noise $\eta$ is independent of $u$. 
As mentioned above, in this paper we present an analysis of ensemble inversion
viewed as a minimisation method applied to the least-squares functional}
\begin{equation}
\label{eq:lsq}
\Phi(u;y^\dagger)=\frac12\|y^\dagger-Au\|_{\Gamma}^2\,,
\end{equation}
\cls{where the norm $\|\cdot \|_\Gamma=\|\Gamma^{-1/2}\cdot\|_2$ corresponds to the Euclidean norm weighted by the square-root of the inverse noise covariance matrix. Accordingly, we define by $\langle\cdot,\cdot\rangle_\Gamma=\langle \Gamma^{-1/2}\cdot,\Gamma^{-1/2}\cdot\rangle$ the corresponding inner product. 
The realisation of the random variable $y$, i.e. the observed data, is denoted by $y^\dagger$. The prior $\mu_0$ plays a role
in the optimisation perspective as the initial ensemble is typically drawn
from $\mu_0$.}

\cls{In order to facilitate analysis we work with continuous time
limit of the ensemble inversion algorithm \cite{AndrewC2017}.
The classic implementation of 
ensemble Kalman inversion, in which the observed data $y^\dagger$ is 
perturbed by the addition of independent draws from the distribution of $\eta$, 
leads to a stochastic differential equation (SDE) limit; the simplification in
which the observed data $y^\dagger$ is unperturbed leads to an ordinary 
differential equation (ODE) in the limit. We work with the ODE limit in this paper.
What distinguishes our analysis from that appearing in \cite{AndrewC2017} is that
we study the case where the observed data $y^\dagger$ appearing in the ODE is assumed
to contain noise-- i.e. it is not simply the image of a truth $u^\dagger$ under $A$; we
refer to this as the noisy, linear setting.} 

The paper is structured as follows.
In Section \ref{sec:EnKF}, we introduce \cls{ensemble Kalman inversion}
and derive the continuous time limit of the algorithm. 
We study the properties of the method by analysing the dynamical behaviour of the ensemble and derive convergence results by considering the long-time behaviour. We present,
in Section \ref{sec:conv}, well-posedness results, quantification of the ensemble collapse and convergence results for the noisy, linear setting. Numerical experiments illustrating the findings are presented in Section \ref{sec:N}.

\section{The \cls{Ensemble Kalman Inversion and its Continuous Time Limit}}\label{sec:EnKF}

The \cls{ensemble inversion method} that we study is given 
in \cite{iglesias2013ensemble}. \cls{By introducing an artificial time $h=1/N$ for a given integer $N$, the method propagates an ensemble $\{u_n^{(j)}\}_{n=0}^N$ of $J$ particles, $J\in \mathbb N$, at discrete time $nh$ into an ensemble at time $(n+1)h$} according to the formula 
\begin{equation}
\label{eq:EnKFiter}
 {u_{n+1}^{(j)}=u_{n}^{(j)}+C(u_{n})A^*(AC(u_{n})A^*+\frac{1}{h}\Gamma)^{-1}(\cls{y^\dagger}- Au_n^{(j)})}. 
\end{equation}
Here
$$\bar u_{n}=\frac1J\sum_{j=1}^J u_{n}^{(j)}, \quad  C(u_{n})=\frac{1}{J}\sum_{j=1}^J\bigl(u^{(j)}_n-\bu_n\bigr)
\otimes\bigl(u^{(j)}_n-\bu_n\bigr).$$ 
The analysis we present here relies on the continuous time limit of \cls{ensemble Kalman inversion}. We therefore interpret the iterates $u_n^{(j)}$ as a discretisation of a continuous function $u^{(j)}(nh)$. In this context the argument for the appearance of
scaling $h^{-1}$ multiplying $\Gamma$ in the update formula is given
in \cite{iglesias2013ensemble}. 
\cls{If we let $h \to 0$ and interpret the iterations as a timestepping scheme, then 
the continuous time limit is given by}
   \begin{equation}\label{eq:lode}
   \frac{\dd u^{(j)}}{\dd t}=\frac{1}{J}\sum_{k=1}^J \bigl\langle A(u^{(k)}-\bu),
\cls{y^\dagger}-Au^{(j)}\bigr\rangle_{\Gamma} \bigl(u^{(k)}-\bu\bigr), \quad j=1,\cdots, J.
  \end{equation}
  or equivalently
  \begin{equation}\label{eq:lode2}   \frac{\dd u^{(j)}}{\dd t}=-C(u)D_{u}\Phi(u^{(j)};\cls{y^\dagger})
\end{equation}
\cls{with potential 
$\Phi(u;y^\dagger)$ given by \eqref{eq:lsq}.}
Equation \eqref{eq:lode} reveals the well-known subspace property of \cls{ensemble Kalman
inversion} \cite{iglesias2013ensemble}, since the vector field is in the linear span of the ensemble itself. \cls{We re-emphasize that the derivation is based on the simplified
version of the classic ensemble Kalman inversion scheme in which perturbations 
of the observed data $y^\dagger$ are set to zero.}

\section{Convergence Analysis}\label{sec:conv}

\cls{This section is devoted to a generalisation of the results} from \cite{AndrewC2017} 
to allow for noise in the observational data; specifically we consider the case that the observational data $y^\dagger$ is polluted by additive noise $\eta^\dagger \in \mathbb R^K$ in the following way:
\[
y^\dagger=Au^\dagger+\eta^\dagger\;,
\] 
where $u^\dagger$ denotes the truth and \cls{$\eta^\dagger$ a realisation of noise.}
In subsection \ref{ssec:3.1}
we will demonstrate the undesirable effect of noise on the inversion methodology,
and in subsection \ref{ssec:3.2}
we will suggest a stopping criterion to ameliorate the effect.

\subsection{Analysis of \cls{Ensemble Kalman Inversion} With Noisy Data}
\label{ssec:3.1}

Following the notation introduced in \cite{AndrewC2017}, we
introduce the quantities
\begin{align}
&e^{(j)}=u^{(j)}-\bar u, \quad r^{(j)}=u^{(j)}-u^\dagger\ \  j=1,\ldots,J\\
&E_{lj}=\langle Ae^{(l)},Ae^{(j)}\rangle_{\Gamma}, \ \
 R_{lj}=\langle Ar^{(l)},Ar^{(j)}\rangle_{\Gamma}, \ \
 F_{lj}=\langle Ar^{(l)},Ae^{(j)}\rangle_{\Gamma}\ \  l,j=1,\ldots,J\,,
\end{align} 
and the misfit $\vartheta^{(j)}=Au^{(j)}-y^\dagger\cls{=A(u^{(j)}-u^\dagger)-\eta^\dagger}, \ j=1,\ldots,J$. 
The quantity $e^{(j)}$ measures, for each particle $j$, the difference to the 
empirical mean (computed from the ensemble) and the quantity $r^{(j)}$ 
measures the difference from particle $j$ to the truth\cls{, i.e. the residuals}. The matrix-valued quantities describe the interaction of these quantities mapped to the observation space. 
Note that the \cls{mapped residuals $Ar^{(j)}=A(u^{(j)}-u^\dagger)$
are related to the misfit by} 
\begin{equation}\label{eq:misfitd}
\vartheta^{(j)}=Ar^{(j)}-\eta^\dagger \ \  j=1,\ldots,J\;;
	\end{equation}
	from this it is apaprent that
the misfit is a finite dimensional quantity in $\mathbb R^K$. Furthermore, we define the matrix-valued quantity $D$ by
	\begin{equation}
	D_{lj}=\langle \vartheta^{(l)}, Ae^{(j)}\rangle_\Gamma\qquad l,j=1,\ldots,J\;.
	\end{equation}

\begin{theorem}\label{t:1pert}
Let $y^\dagger$ denote the perturbed image of a truth \cls{$u^\dagger \in \cX: 
y^\dagger=Au^\dagger+\eta^\dagger$} for some $\eta^\dagger \in \mathbb R^K$. 
Furthermore, an initial ensemble $u^{(j)}(0) \in \cls{\cX}$
for $j=1,\dots, J$ is given, and we denote by $\cls{\mathcal X_0}$ the linear span of the $\{u^{(j)}(0)\}_{j=1}^J.$ Then, equation \eqref{eq:lode} has a unique solution 
$u^{(j)}(\cdot) \in C([0,T);\cls{\mathcal X_0})$ for $j=1,\dots, J.$
\end{theorem}
\begin{proof}
\cls{The preservation of $\mathcal X_0$ by the ensemble Kalman iteration, and its
continuous time limit, is not affacted by the presence of noise in the data
$y^dagger.$ Each particle $u^{(j)}$ satisfies}
\begin{eqnarray}
\frac{\dd  u^{(j)}}{\dd t}&=& -\frac{1}{J}\sum_{k=1}^J F_{jk}e^{(k)}+\frac1J\sum_{k=1}^J\cls{\langle \eta^\dagger, Ae^{(k)}\rangle_\Gamma} e^{(k)}\nonumber\\
&=& \cls{-\frac{1}{J}\sum_{k=1}^J D_{jk}e^{(k)}}\nonumber\\
&=& -\frac{1}{J}\sum_{k=1}^J D_{jk}u^{(k)}\label{eq:uD}\,.
\end{eqnarray}
\cls{We have used the fact that $\sum_{k=1}^J D_{jk}=0$, thus $\sum_{k=1}^J D_{jk}\bar u=0$.}
\cls{The preservation of  $\mathcal X_0$} and the local Lipschitz continuity of the right-hand side of \eqref{eq:uD} ensures the local existence of a solution in $C([0,T);\cls{\mathcal X_0})$ for $T>0$.
To establish global existence of solutions, we now 
show the boundedness of the right-hand side of \eqref{eq:uD}. 

The following differential equation holds for the quantity $e^{(j)}$: 
\begin{equation}
\label{eq:e}
\frac{\dd  e^{(j)}}{\dd t}= -\frac{1}{J}\sum_{k=1}^J E_{jk}e^{(k)}=-\frac{1}{J}\sum_{k=1}^J E_{jk}r^{(k)}\,.
\end{equation}
For the matrix-valued quantity $E$, we obtain
\begin{eqnarray*}
   \frac{\dd  }{\dd t}E=-\frac2JE^2\,.
     \end{eqnarray*}
Thus, the dynamical behaviour of the quantities $e^{(j)}$ and $Ae^{(j)}$ is not influenced by the noise in the data. Therefore, the results presented in \cite{AndrewC2017} for the noise free case still hold: 
for the orthogonal matrix $X$ defined through
the eigendecomposition of $E(0)$ it follows that 
       \begin{eqnarray} \label{eq:diagE}
  E(t)=X\Lambda(t)X^{{\top}}\;
     \end{eqnarray}
     with
$\Lambda(t)=\mbox{diag}\{ \lambda^{(1)}(t),\ldots,\lambda^{(J)}(t)\}$,
$\Lambda(0)=\mbox{diag}\{ \lambda_0^{(1)},\ldots,\lambda_0^{(J)}\}$
and
            \begin{eqnarray} \label{eq:odeenslambda}
   \lambda^{(j)}( t)=\Big({\frac 2J t+\frac{1}{\lambda_0^{(j)}}}\Big)^{-1}\;,
     \end{eqnarray}
     if $ \lambda_0^{(j)}\neq 0$, otherwise $ \lambda^{(j)}(t)= 0$.
\cls{This proves that the matrix $E$, and hence all its elements, are globally
bounded in time.}
     
The misfit $\vartheta^{(j)}$ satisfies 
\begin{equation}
\frac{\dd  \vartheta^{(j)}}{\dd t}= -\frac{1}{J}\sum_{k=1}^J D_{jk}Ae^{(k)}
\end{equation}
and the dynamical behaviour of the corresponding matrix-valued quantity $D$ is given by
\begin{equation}
\frac{\dd }{\dd t}D= -\frac{2}{J}DE\;.
\end{equation}
The boundedness of $D(t)$ follows from the boundedness of the misfit $\vartheta^{(j)}$, which can be derived from
  \begin{equation}
\cls{\frac12\frac{\dd  \|\vartheta^{(j)}\|_\Gamma^2}{\dd t}= -\frac{1}{J}\sum_{k=1}^J D_{jk}D_{jk}}\,.
\end{equation}
Hence, the misfit $\vartheta^{(j)}$ is bounded uniformly in time. By the \cls{Cauchy-Schwarz} inequality, the bound on $D$ follows with
 \begin{eqnarray*}
 \cls{D_{ij}^2=\langle \vartheta^{(i)},Ae^{(j)}\rangle_{\Gamma}^2\le \|\vartheta^{(i)}\|_{\Gamma}^2\cdot \|Ae^{(j)}\|_{\Gamma}^2 \le C \|Ae^{(j)}\|_{\Gamma}^2\;}
 \end{eqnarray*}
 for a constant $C>0$ independent of $T$. \cls{This establishes that
 $D_{ij} \rightarrow 0$ at least as fast as {$\frac{1}{\sqrt{t}}$} as $t \rightarrow \infty$}, in particular, $D$ is uniformly bounded in time. \cls{Note that the convergence rate follows from the convergence rate $1$ of the quantity $ \|Ae^{(j)}\|_{\Gamma}^2$
established in \eqref{eq:odeenslambda}.}
Global existence for $u^{(j)}$ (and $e^{(j)}$, $r^{(j)}$)
follows.
\end{proof}
The proof of Theorem \ref{t:1pert} reveals that the behaviour of the quantity $e^{(j)}$, which is an indicator of the ensemble collapse, is not affected by the noise. Hence, \cite[Theorem 3]{AndrewC2017} can be directly generalised to the perturbed case.
\begin{corollary}\label{cor:enscoll}
Let $y^\dagger$ denote the perturbed image of a 
truth \cls{$u^\dagger \in \cX:
y^\dagger=Au^\dagger+\eta^\dagger$ for some $\eta^\dagger \in \mathbb R^K$.}
Furthermore, assume that an initial ensemble $u^{(j)}(0) \in \cls{\cX}$
for $j=1,\dots, J$ is given. Then, the matrix valued quantity $E(t)$ converges to $0$ for $t \to \infty $ with an algebraic rate of convergence: {$\|E(t)\|={\mathcal O}(Jt^{-1}).$} 
\end{corollary}

The ensemble collapse is a further form of regularisation as the solution
not only remains in the linear span of the initial ensemble, but actually
asymptotically lives in the span of a single element\cls{, provided that the forward response operator $A$ is one-to-one} .
The preceding result shows that the ensemble collapse, namely the fact
that all particles converge to their common mean, does not depend on the 
realisation of the noise. 
We now discuss the convergence properties of \cls{ensemble Kalman inversion} in the noisy case. The analysis presented in \cite[Theorem 4]{AndrewC2017} indicates that we can transfer the convergence result straightforwardly to the mismatch $\vartheta^{(j)}$. However, the convergence of the residuals $r^{(j)}$ depends on the realisation of the noise.

\begin{theorem}
\label{t:3pert}
Let $y^\dagger$ denote the \cls{noisy image of a 
truth $u^\dagger \in \cX:
y^\dagger=Au^\dagger+\eta^\dagger$ for some $\eta^\dagger \in \mathbb R^K$.}
\cls{Assume further that the forward operator $A$ is one-to-one}. 
Let {$\Yl$} denote the linear span of the $\{{A}e^{(j)}(0)\}_{j=1}^J$
and let {$\Yp$} denote the orthogonal complement of {$\Yl$}
in \cls{$\mathbb R^K$} and assume that the initial ensemble members are chosen
so that {$\Yl$} has the maximal dimension {$\min\{J-1,\dim(Y)\}.$}  
Then $\vartheta^{(j)}(t)$ may be decomposed uniquely
as {$\dll(t)+\dpp(t)$ with $\dll \in \Yl$ and $\dpp \in \Yp$}, where ${\dll}(t) \to 0$ as $t \to \infty$ and $\dpp(t)=\dpp(0)={\vartheta^{(1)}_{\perp}}.$

Furthermore, if $\langle \eta^\dagger, Ae^{(k)}\rangle \le \langle Ar^{(k)}, Ae^{(k)}\rangle$, the \cls{mapped residual} is monotonically decreasing. The rate of convergence of the component of the residual mapped forward to the observational space, which belongs to $\Yl$, can be arbitrarily slow, i.e. depending on the realisation of the noise, the rate of convergence can be arbitrarily close to $0$.

\end{theorem}
\begin{proof}
The first part of the theorem follows with the same arguments as used for the proof  of \cite[Theorem 4]{AndrewC2017}. \cls{For the second part we observe that the norm of 
the \cls{mapped} residuals satisfies the following differential equation:}
\begin{equation}
\frac12 \frac{\dd }{\dd t}\|Ar^{(j)}\|_\Gamma^2=-\frac1J \sum_{k=1}^J F_{jk}^2+\frac1J\sum_{k=1}^J \langle Ar^{(k)},Ae^{(k)}\rangle_\Gamma\langle \eta^\dagger,Ae^{(k)}\rangle_\Gamma\;.
\end{equation}
Provided that $\langle \eta^\dagger, Ae^{(k)}\rangle \le \langle Ar^{(k)}, Ae^{(k)}\rangle_\Gamma$ for $k=1,\ldots,J$, i.e. $\|\eta^\dagger\|_\Gamma \cos(\theta_1)\le \cls{\|Ar^{(k)}\|_\Gamma} \cos(\theta_2)$ with $\theta_1$ and $\theta_2$ denoting the angle between $\eta^\dagger$ and $Ae^{(k)}$, and between $Ar^{(k)}$ and $Ae^{(k)}$, respectively, the residuals mapped to the image space of the forward operator are monotonically decreasing.
Expanding the quantities $Ar^{(k)}$ and $\eta^\dagger$ in $\Yl$ and the orthogonal complement $\Yp$

\begin{eqnarray*}
{A}r^{(j)}(t)=\sum_{k=1}^J\alpha_k {A}e^{(k)}(t)+{Ar^{(1)}_{\perp}}\\
\eta^\dagger=\sum_{k=1}^J\eta_k {A}e^{(k)}(t)+{A\eta^{(1)}_{\perp}}\;,
\end{eqnarray*}
cp. \cite[Lemma 8]{AndrewC2017} yields
\begin{eqnarray*}
\frac12 \frac{\dd }{\dd t}\|Ar^{(j)}\|_{\Gamma}^2&=&-\frac1J \sum_{k=1}^J\sum_{l=1}^J E_{lk}\alpha_k E_{kl}\alpha_l+ \frac1J \sum_{k=1}^J\sum_{l=1}^J E_{lk}\alpha_k E_{kl}\eta_l\;. 
\end{eqnarray*}
If the coefficients of the noise are of the size of $\alpha_k$, the right hand side becomes zero and the claim follows.
\end{proof}


{

\subsection{Stopping Criteria for \cls{Ensemble Kalman Inversion}}
\label{ssec:3.2}

The Bayesian \cls{derivation of the ensemble Kalman inversion algorithm
given in \cite{AndrewC2017} suggests an integration of the limiting equation
\eqref{eq:lode} up to time $T=1$.} This can be interpreted as an a priori regularisation strategy motivated by the probabilistic viewpoint. However, this stopping rule does not take into account the actual realisation of the noise nor the additional regularisation effect due 
to the ensemble collapse. \cls{Indeed ensemble collapse is caused by removing 
random noisy perturbations within the algorithm, causing an underestimation of 
the variance for linear Gaussian problems, suggesting that stopping at time $T=1$
may no longer be the right choice as the Bayesian connection can no longer be
justified. Our numerical experiments will indeed show that the
Bayesian stopping strategy often leads to a stopping criterion for the
unperturbed algorithm which is too early.}

\cls{The papers \cite{iglesias2014iterative,iglesias2015regularizing} suggest
an approach to regularising discrete-time ensemble Kalman inversion
methods, based on an analogy with deterministic iterative methods such as
Levenberg-Marquardt. Unfortunately this methodology does not transfer directly
to our continuous time setting as it corresponds to an adaptive time step, rather
than the fixed time-step $h$ used in the derivation above.
The proof of Theorem \ref{t:3pert} suggests an a posteriori stopping criterion 
for the method. In the deterministic setting, Morozov's discrepancy principle is a widely used and well understood stopping rule, see \cite{engl1996regularization} and
the references therein. The idea of this stopping rule is that, due to noisy data, the information in the observations cannot be distinguished from the noise for a \cls{mapped }residual which is on the order of the noise level $\delta$. This suggests that
asking for a \cls{mapped }residual with discrepancy smaller than $\delta$ may
lead to fitting of the unknown parameters to the noise.  
 We will numerically investigate the discrepancy principle as a suitable 
criterion in the presented setting. Furthermore, we note that if the noise is orthogonal to the space 
spanned by the linear ensemble, then Theorem \ref{t:3pert} 
shows the convergence of the \cls{mapped }residuals in the image space.

Motivated by the deterministic regularisation methods, the discrepancy principle is generalised to statistical noise; see \cite{Kaipio} for example. 
The iterations of the iterative ensemble method will be stopped when
\begin{equation}\label{eq:discrep}
\|A\bar u(t)-y^\dagger\|_2\le\tau\sqrt{{\rm trace}{(\Gamma)}}
\end{equation}
where $\tau>1$ is a given parameter and $\bar u(t)$ denotes the empirical mean of the ensemble at artificial time $t$. Here, the average noise level $\mathbb E(\|\eta\|_2^2)={\rm trace}{(\Gamma)}$ is taken into account.
(Since the noise in the observations is assumed to be normally distributed realisations of the noise cannot be bounded from above and below.)

The discrepancy principle for statistical noise \eqref{eq:discrep} does not generalise to the infinite or high-dimensional setting, as the residual is no longer a well-defined quantity. In \cite{Blanchard}, symmetrisation is suggested to overcome this problem
leading to  the stopping criterion
\[
\|A^*\Gamma^{-1}A\bar u(t)-A^*\Gamma^{-1/2}y^\dagger\|\le\tau\sqrt{{\rm trace}{(A^*\Gamma^{-1}A)}}.
\]
In order to obtain optimal rates, the authors in \cite{Blanchard} suggest modifying
this discrepancy principle to
\[
\|(\lambda I+A^*\Gamma^{-1}A)^{-1/2}(A\bar u(t)-y^\dagger)\|\le\tau\sqrt{{\rm trace}{((\lambda I+A^*\Gamma^{-1}A)^{-1}A^*\Gamma^{-1}A)}}\,, 
\]
where $\lambda>0$ is a given fixed parameter. The analysis presented in \cite{Blanchard} proving optimality of the strategy is not directly applicable to the ensemble
Kalman inversion methodology that we study here, due to the nonlinear nature of the
ensemble algorithms. However, we will observe in the numerical experiments that the modified version of the discrepancy principle leads to satisfactory results. 

We also remark that stopping strategies taking into account the Bayesian viewpoint on the inverse problem lead to appealing alternatives. Assuming a Gaussian prior distribution for example, the parametrised variance can be modelled as a hyperparameter, which can then be estimated from the data. Due to the regularisation effect of the ensemble, this can be viewed as an alternative stopping / regularisation strategy. Closely related is the idea of variance inflation, which can be interpreted in a similar way. The work presented here, however, is restricted to the deterministic setting, not taking into account the Bayesian viewpoint. The analysis of the stopping rules requires therefore a different setting, which is beyond the scope of the paper.
 

}

\section{Numerical Experiments}\label{sec:N}
The forward model is described by the one dimensional elliptic equation
 \begin{equation}\label{eq:forward}
    -\frac{\rd^2 p}{\rd x^2}+p=u \quad \mbox{in } D:=(0,\pi)\, , \ p=0  \quad \mbox{in } \partial D\; .
\end{equation}
The solution operator of the model is a mapping \cls{$G:L^2(D)\to H^2(D) \cap H^1_0(D)$} taking
$u$ into $p.$ The solution is observed at $K=2^{4}-1$ equispaced observation points at $x_k=\frac{k}{2^{4}}, k=1,\ldots,2^{4}-1$, which defines the observation operator \cls{$\cO:H^2(D) \cap H^1_0 \to \mathbb R^K$}, i.e. the operator $A$ is a mapping from \cls{$L^2(D)$} to $\mathbb R^K$ defined by the composition of the solution operator and the 
(pointwise) observation operator.
We use a finite element method with continuous, piecewise linear ansatz functions on a uniform mesh with meshwidth $h=2^{-8}$ to solve the forward problem (the spatial discretisation leads to a discretisation of $u$, i.e. $u \in \mathbb R^{2^8-1}$).

Then, the inverse problem consists of recovering the unknown data $u$ from noisy observations
\begin{eqnarray}\label{eq:modelne}
y^\dagger&=&\cls{\cO(p)}+\eta^\dagger =Au^\dagger+\eta^\dagger\;.
\end{eqnarray}
The measurement noise is chosen to be normally distributed, $\eta \sim \mathcal N(0,\gamma I )$, $\gamma=0.01^2 \in \mathbb R, \ I \in \mathbb R^{K \times K}$. Furthermore, the prior is $\mu_0=N(0,C_0)$ with covariance operator $C_0=10(-\Delta)^{-1}$. Here, we consider the Laplacian $\Delta$ with domain $H^2(D) \cap H^1_0(D)$. The initial ensemble is based on the eigendecomposition of the covariance operator $C_0$, i.e. $u^{(j)}(0)=\sqrt{\lambda_j}\zeta_j z_j$ with $\zeta_j \sim \mathcal N(0,1)$ for $j=1,\ldots,J$ and $\{\lambda_j,z_j\}_{j \in \mathbb N}$ denoting \cls{(the explicitly known)} eigenvalues and eigenfunctions 
of $C_0$. 

To illustrate and numerically verify the results presented in this paper, we investigate the dynamical behaviour of the quantities $e,r$ and the misfit $\vartheta$. 
The theoretical results presented hold true \cls{for any ensemble size}; we consider in the following a rather small ensemble of size $J=5$. For the sake of presentation, the empirical mean (and minimum and maximum deviations) of the ensemble is shown.
\begin{center}
\begin{minipage}{0.6\textwidth}
 \begin{figure}[H]
\centering
    \includegraphics[width=1.0\textwidth]{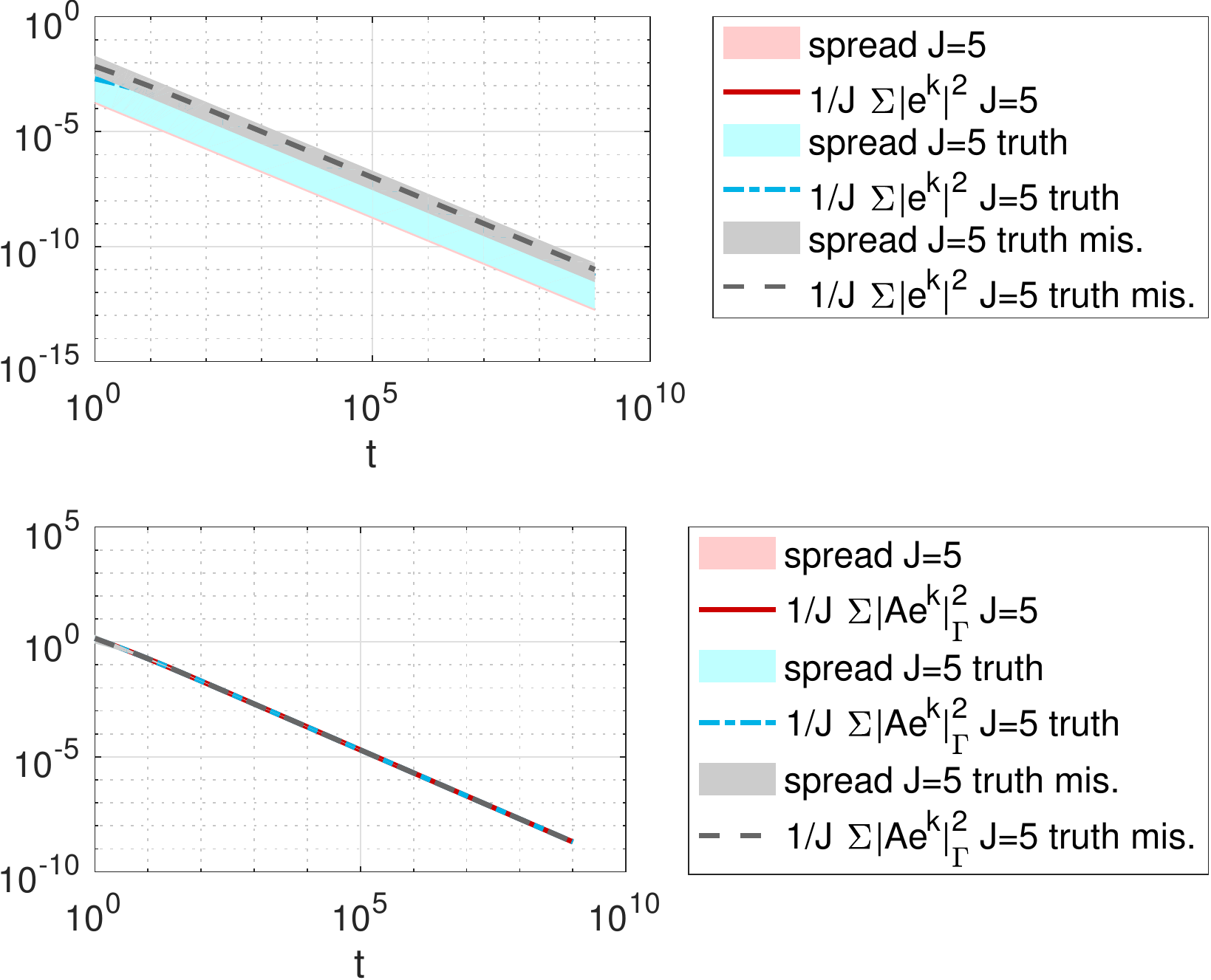}
~\\[-0.75cm]\caption{\footnotesize \label{fig:linNFe}
Quantities $|e|_2^2$, $|Ae|_{\Gamma}^2$ w.r. to time $t$, $J=5$ (KL red), ($u^\dagger$ adaptive blue), ($\tilde u$ adaptive gray),  $K=2^4-1$, initial ensemble chosen based on KL expansion of ${C_0}=10(-\Delta)^{-1}$, $\eta\sim{\mathcal N}(0,0.01^2 \id)$.}
 \end{figure}
\end{minipage}
\end{center}
To investigate the convergence results further, we compare the performance of three ensembles (all of size $J=5$): the first one (shown in red) is based on the first five terms in the Karhunen-Lo\`eve (KL) expansion of the covariance operator $C_0$, the second one (shown in blue) is chosen such that the contribution of $A\rp(t)$ in Theorem \ref{t:3pert} is minimised (i.e. $Ar^{(1)}=\sum_{k=1}^J \alpha_k Ae^{(k)}$ for some coefficients $\alpha_k\in\mathbb R$. Given $u^{(2)},\ldots,u^{(J)}$ and coefficients $\alpha_1,\ldots,\alpha_J$, we define $u^{(1)}=(1-\alpha_1+\sum_{k=1}^J \alpha_k/J)^{-1}(u^\dagger-\alpha_1/J\sum_{j=2}^J\ u^{(j)}+\sum_{k=2}^J\alpha_k u^{(k)}-\alpha_k/J \sum_{j=2}^J u^{(j)})$), the third ensemble (shown in grey) is chosen such that the contribution of $\vartheta(t)_{\perp}$ in Theorem \ref{t:3pert} is minimised (i.e. $\vartheta^{(1)}=\sum_{k=1}^J \alpha_k Ae^{(k)}$ for some coefficients $\alpha_k\in\mathbb R$. Given $u^{(2)},\ldots,u^{(J)}$ and coefficients $\alpha_1,\ldots,\alpha_J$, we define $u^{(1)}=(1-\alpha_1+\sum_{k=1}^J \alpha_k/J)^{-1}(\tilde u-\alpha_1/J\sum_{j=2}^J\ u^{(j)}+\sum_{k=2}^J\alpha_k u^{(k)}-\alpha_k/J \sum_{j=2}^J u^{(j)})$, where $\tilde u$ is the minimiser of the underdetermined least-squares problem).

In practice, the second strategy is not implementable, since the truth is used to construct the ensemble. However, the performance of the second strategy gives useful insight into the convergence behaviour of \cls{ensemble Kalman inversion.}  

The ensemble collapse is not affected by the choice of the initial ensemble. We observe the predicted algebraic rate of convergence to the empirical mean, cp Figure \ref{fig:linNFe}.

~\\

The convergence behaviour of the \cls{mapped }residuals and the misfit, both projected to the subspace spanned by the initial ensemble and the complement are shown in the \cls{
Figures \ref{fig:linNDmis} and \ref{fig:linNDres}. }
~\\
 \begin{minipage}{0.45\textwidth}
 \begin{figure}[H]
\centering
    \includegraphics[width=\textwidth]{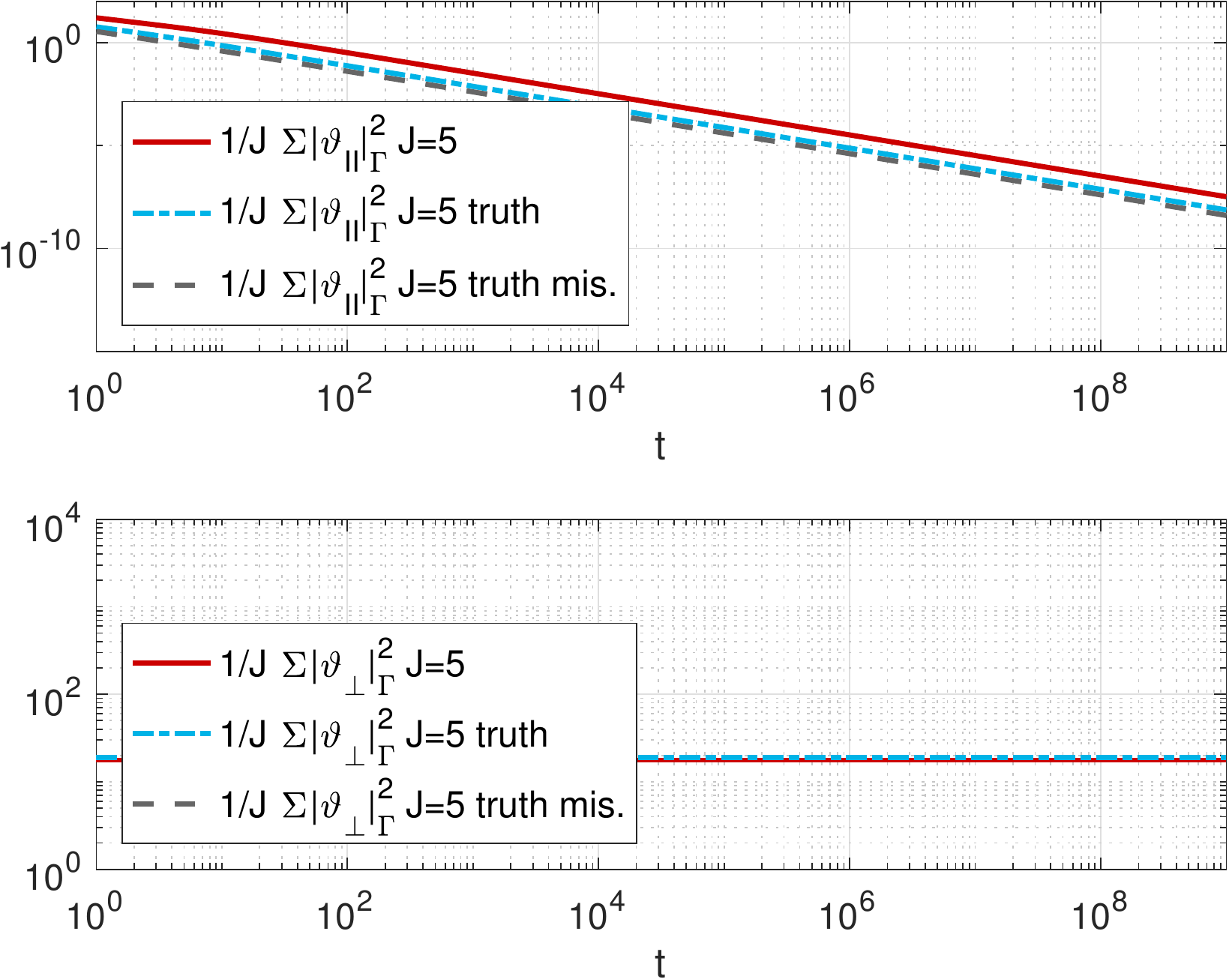}
~\\[-0.75cm]\caption{\footnotesize\label{fig:linNDmis}
Misfit $|\vartheta_{II}|_\Gamma^2$ and $|\vartheta_{\perp}|_\Gamma^2$ w.r. to time $t$, $J=5$ (KL red), ($u^\dagger$ adaptive blue), ($\tilde u$ adaptive grey),  $K=2^4-1$, initial ensemble chosen based on KL expansion of ${C_0}=10(-\Delta)^{-1}$, $\eta\sim{\mathcal N}(0,0.01^2 \id)$.}
 \end{figure}
\end{minipage}
~
\begin{minipage}{0.45\textwidth}
 \begin{figure}[H]
\centering
    \includegraphics[width=\textwidth]{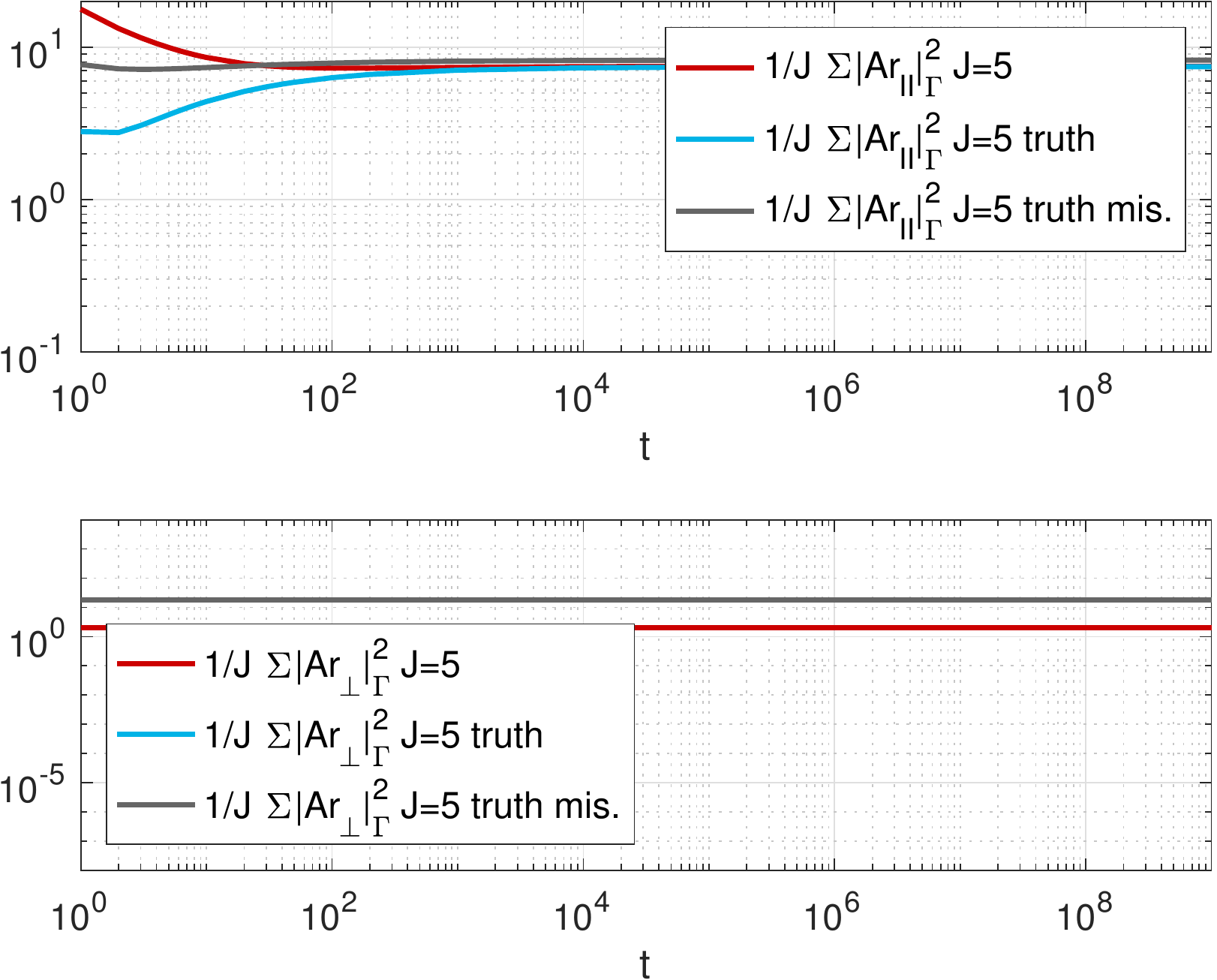}
~\\[-0.75cm]\caption{\footnotesize\label{fig:linNDres}
Mapped residuals $|Ar_{II}|_\Gamma^2$ and $|Ar_{\perp}|_\Gamma^2$ w.r. to time $t$, $J=5$ (KL red), ($u^\dagger$ adaptive blue), ($\tilde u$ adaptive grey),  $K=2^4-1$, initial ensemble chosen based on KL expansion of ${C_0}=10(-\Delta)^{-1}$, $\eta\sim{\mathcal N}(0,0.01^2 \id)$.}
 \end{figure}
\end{minipage}

~\\

The algebraic rate of the misfit is clearly confirmed. Furthermore, the convergence behaviour of the \cls{mapped }residuals for the KL based ensemble (shown in red in \cls{Figure \ref{fig:linNDres}}) illustrates the arbitrarily slow convergence predicted by the theory, i.e. we observe a convergence rate deteriorating to $0$. For the other two ensembles, we even observe an increase in the \cls{mapped} residual, since the angle conditions are not satisfied.
The comparison of the resulting estimates with the truth reveals the strong overfitting effect of the third ensemble, cp Figure \ref{fig:linNDsol}. This behaviour is expected due to the construction of the ensemble, which implies an amplification of the noise in the data.
\begin{center}
 \begin{minipage}{0.6\textwidth}
 \begin{figure}[H]
\centering
~\\[0.2cm]
    \includegraphics[width=\textwidth]{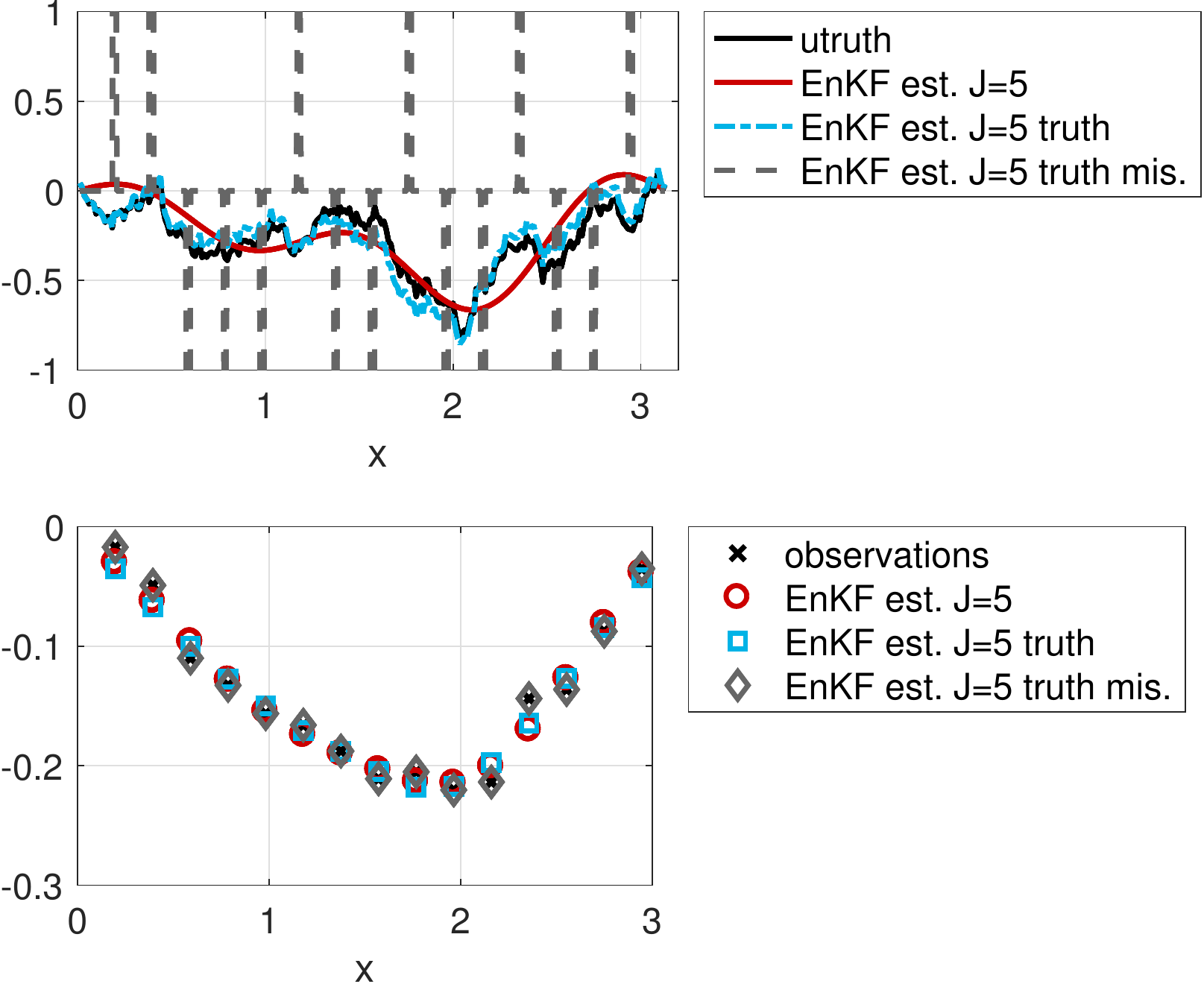}
    ~\\[-0.75cm]\caption{\footnotesize \label{fig:linNDsol}
Comparison of the \cls{ensemble Kalman inversion} estimate with the truth and the observations, $J=5$ (KL red), ($u^\dagger$ adaptive blue), ($\tilde u$ adaptive grey),  $K=2^4-1$, initial ensemble chosen based on KL expansion of ${C_0}=10(-\Delta)^{-1}$, $\eta\sim{\mathcal N}(0,0.01^2 \id)$.}
 \end{figure}
 \end{minipage}
 \end{center}
 
 ~\\
 
 To illustrate the effect of the angle condition and the resulting degradation of the convergence order of the \cls{mapped }residuals, we repeat the experiments with noise in the data, which is orthogonal to the subspace spanned by the initial ensemble. The theoretical results suggest an algebraic rate of convergence, which can be confirmed by the results presented in Figure \ref{fig:linNDmisT}.
 \begin{center}
  \begin{minipage}{0.6\textwidth}
 \begin{figure}[H]
\centering
    \includegraphics[width=\textwidth]{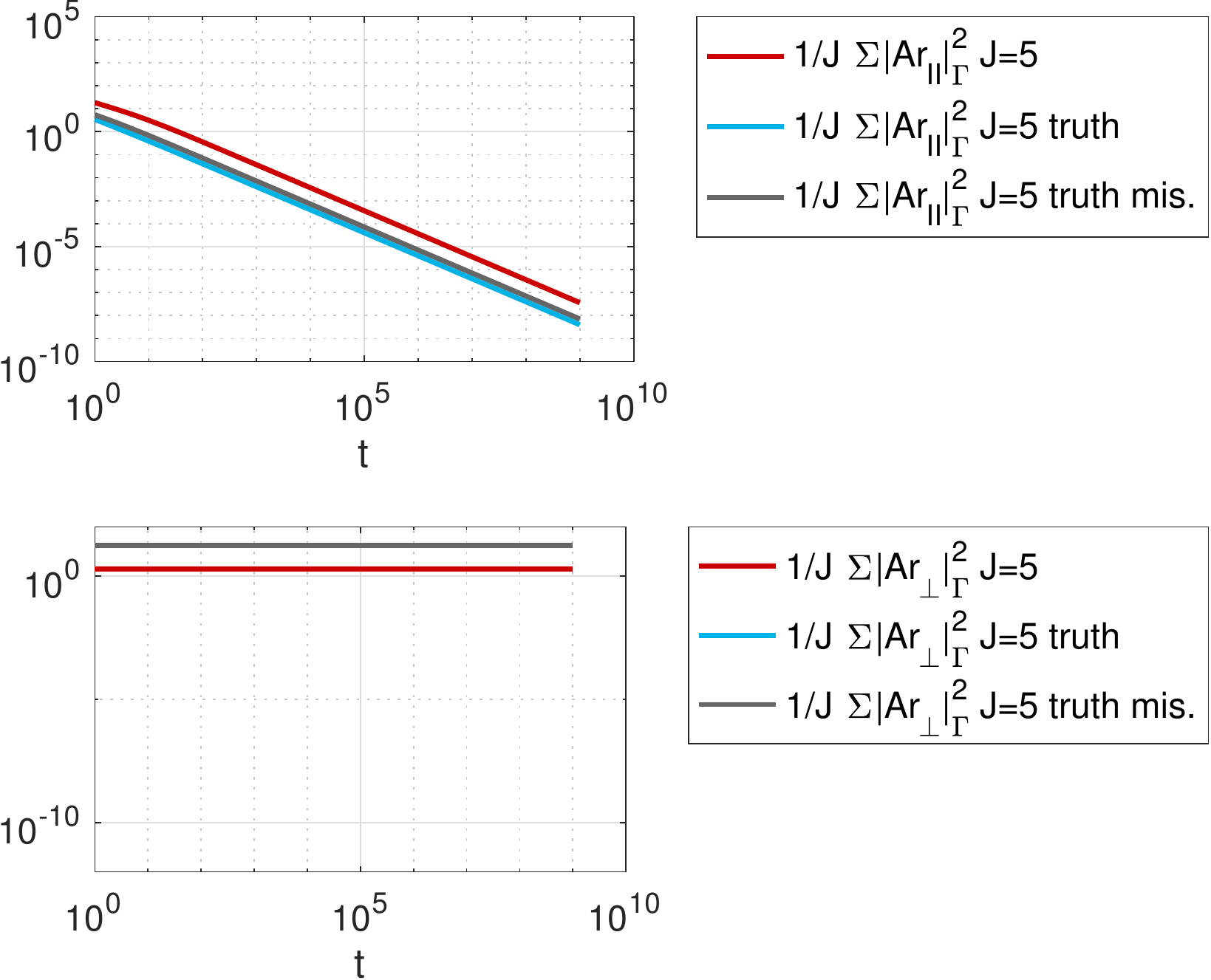}
~\\[-0.75cm]\caption{\footnotesize\label{fig:linNDmisT}
Mapped residuals $|Ar_{II}|_\Gamma^2$ and $|Ar_{\perp}|_\Gamma^2$ w.r. to time $t$, $J=5$ (KL red), ($u^\dagger$ adaptive blue), ($\tilde u$ adaptive grey),  $K=2^4-1$, initial ensemble chosen based on KL expansion of ${C_0}=10(-\Delta)^{-1}$, $\eta\sim{\mathcal N}(0,0.01^2 \id)$, observational noise orthogonal to the subspace spanned by the initial ensembles.}

 \end{figure}
\end{minipage}
\end{center}

~\\

The result on the ensemble collapse Corollary \ref{cor:enscoll} indicates that the regularisation effect of the method strongly depends on the number of particles in the ensemble. The Bayesian stopping rule, which can be interpreted as an a priori stopping rule, does not reflect this behaviour. 
The results of the Bayesian stopping rule are summarised in Figures \ref{fig:disc5solBayes}-\ref{fig:disc50Bayes}.

 \begin{minipage}{0.45\textwidth}
 \begin{figure}[H]
\centering
     \includegraphics[width=\textwidth]{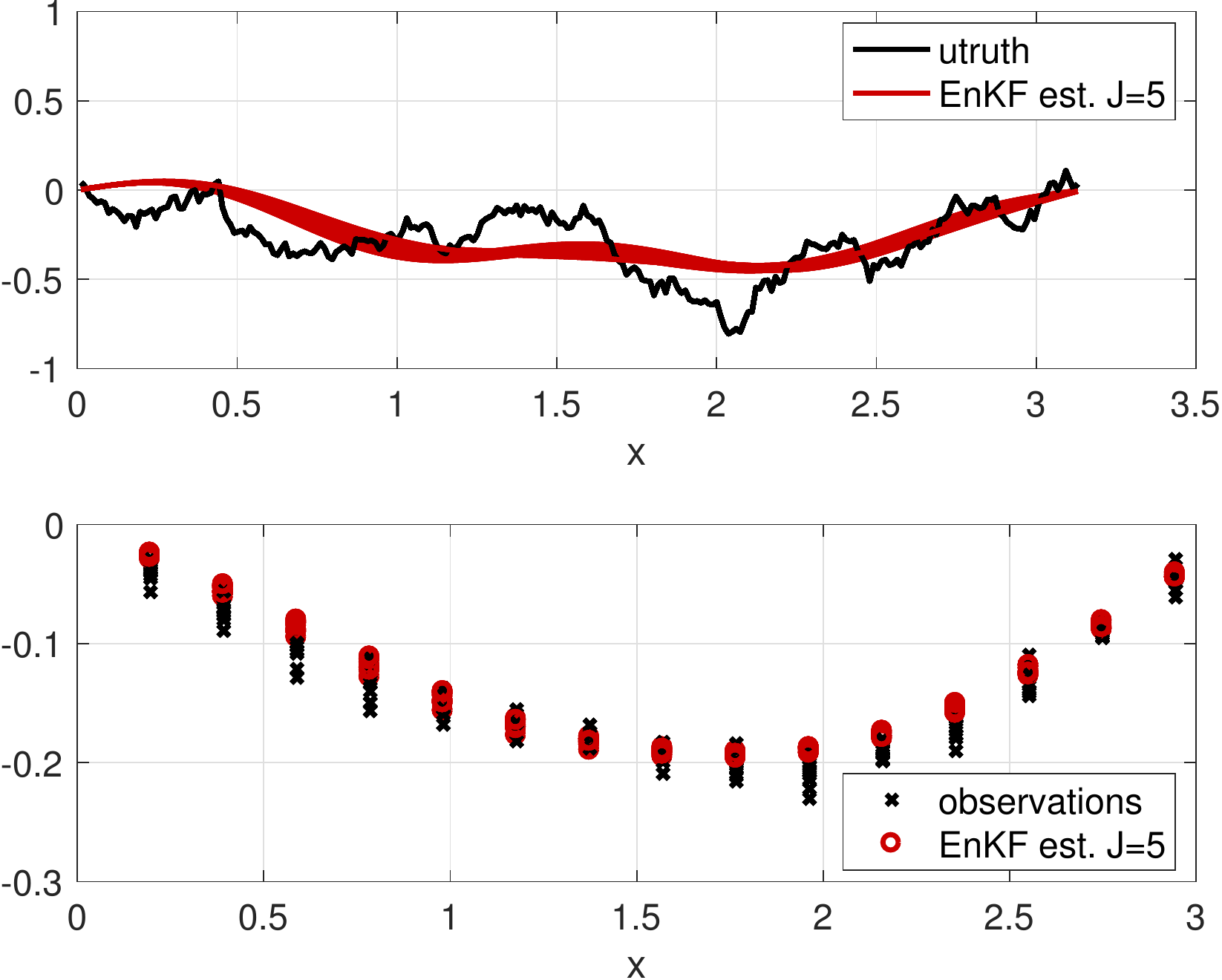}
    ~\\[-0.75cm]\caption{\footnotesize \label{fig:disc5solBayes}
Comparison of the \cls{ensemble Kalman inversion} estimate with the truth and the observations with Bayesian stopping rule, $J=5$ based on KL expansion  of ${C_0}=10(-\Delta)^{-1}$ (red), $K=2^4-1$, $10$ randomly initialised ensembles, $10$ randomly perturbed observation.}
 \end{figure}
 \end{minipage}
~
\begin{minipage}{0.45\textwidth}
 \begin{figure}[H]
\centering
     \includegraphics[width=\textwidth]{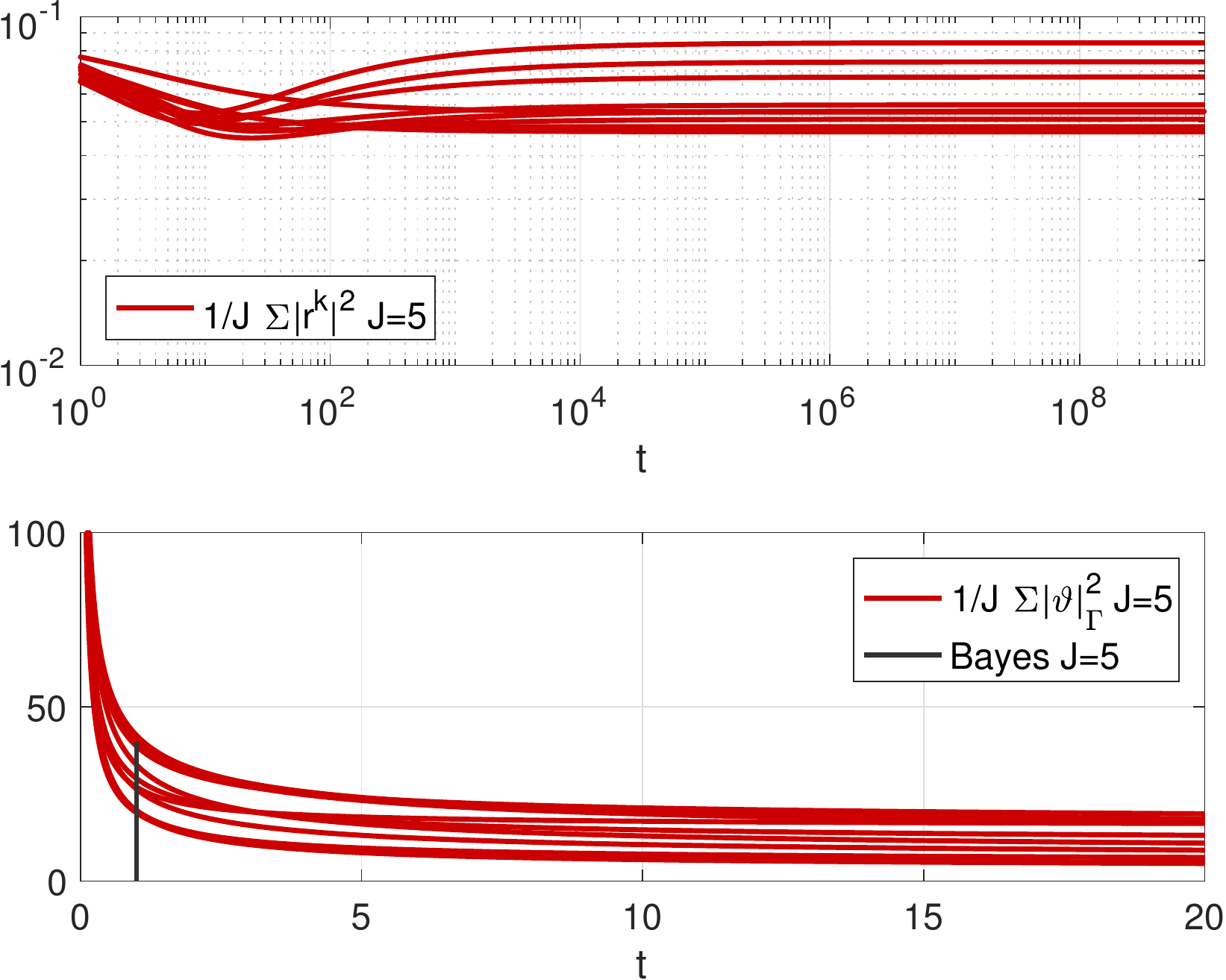}
~\\[-0.75cm]\caption{\footnotesize\label{fig:disc5Bayes}
Mapped residuals $|A\bar r|_\Gamma^2$ (above) and $|\bar \vartheta|_\Gamma^2$ with Bayesian stopping rule (below) w.r. to time $t$, $J=5$ based on KL expansion  of ${C_0}=10(-\Delta)^{-1}$ (red),  $K=2^4-1$, $10$ randomly initialised ensembles, $10$ randomly perturbed observation.}
 \end{figure}
\end{minipage}

 \begin{minipage}{0.45\textwidth}
 \begin{figure}[H]
\centering
     \includegraphics[width=\textwidth]{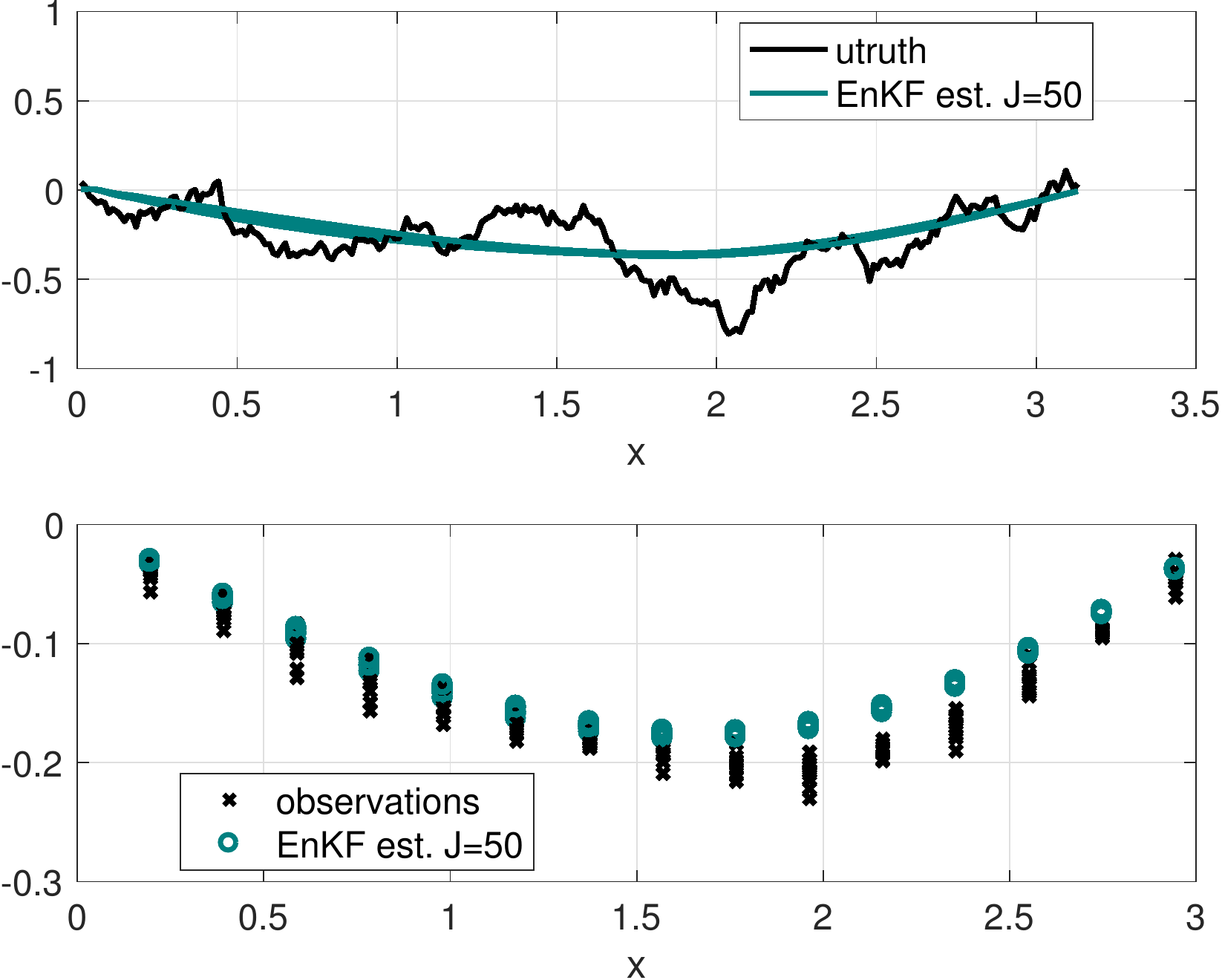}
    ~\\[-0.75cm]\caption{\footnotesize \label{fig:disc50solBayes}
Comparison of the \cls{ensemble Kalman inversion} estimate (with stopping rule) with the truth and the observations with Bayesian stopping rule, $J=50$ based on KL expansion  of ${C_0}=10(-\Delta)^{-1}$ (blue), $K=2^4-1$, $10$ randomly initialised ensembles, $10$ randomly perturbed observation.}
 \end{figure}
 \end{minipage}
~
\begin{minipage}{0.45\textwidth}
 \begin{figure}[H]
\centering
     \includegraphics[width=\textwidth]{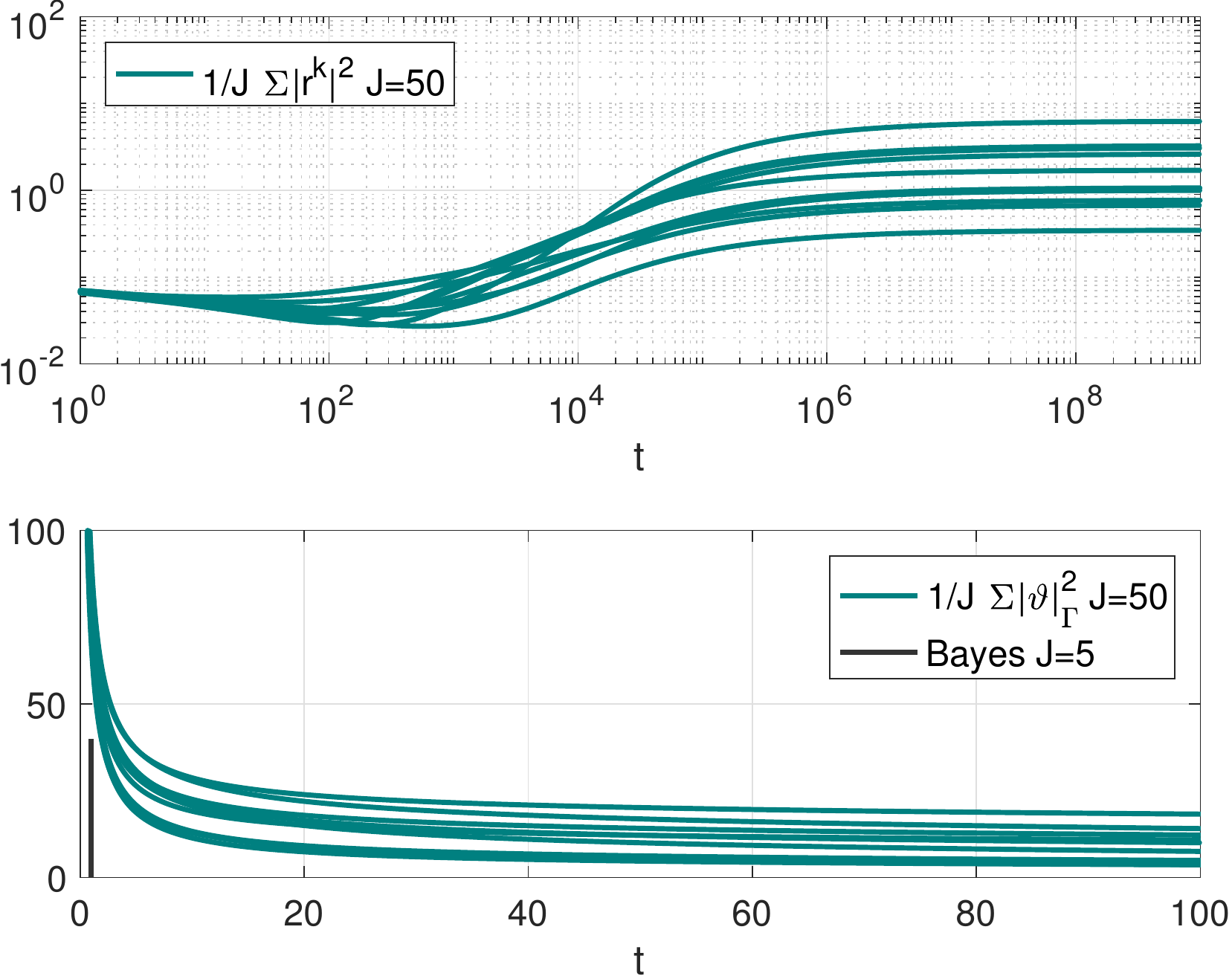}
~\\[-0.75cm]\caption{\footnotesize\label{fig:disc50Bayes}
Mapped residuals $|A\bar r|_\Gamma^2$ (above) and $|\bar \vartheta|_\Gamma^2$ with Bayesian stopping rule (below) w.r. to time $t$, $J=50$ based on KL expansion  of ${C_0}=10(-\Delta)^{-1}$ (blue),  $K=2^4-1$, $10$ randomly initialised ensembles, $10$ randomly perturbed observation.}
 \end{figure}
\end{minipage}

~\\
We will show in the following that the discrepancy principle leads to suitable stopping strategy, in particular, it has the potential to substantially improve the accuracy of the \cls{ensemble Kalman inversion} estimate. To do so, we repeat the experiments with $10$ randomly chosen ensembles (based on the KL expansion of the prior covariance operator) of size $J=5$ and $J=50$. The noise in the data is randomly chosen from $ \mathcal N(0,\gamma I )$ with $\gamma=0.01^2\in\cls{\mathbb R}$. Motivated by the previous discussion on the discrepancy principle, we implement a stopping rule of the form $
\|A\bar u(t)-y^\dagger\|_2 \le 1.2 \sqrt{K}\gamma$, where $K$ denotes the number of observations. Figures \ref{fig:disc5sol} - \ref{fig:disc50} show the comparison of the estimates based on the discrepancy principle with the truth.

 \begin{minipage}{0.45\textwidth}
 \begin{figure}[H]
\centering
     \includegraphics[width=\textwidth]{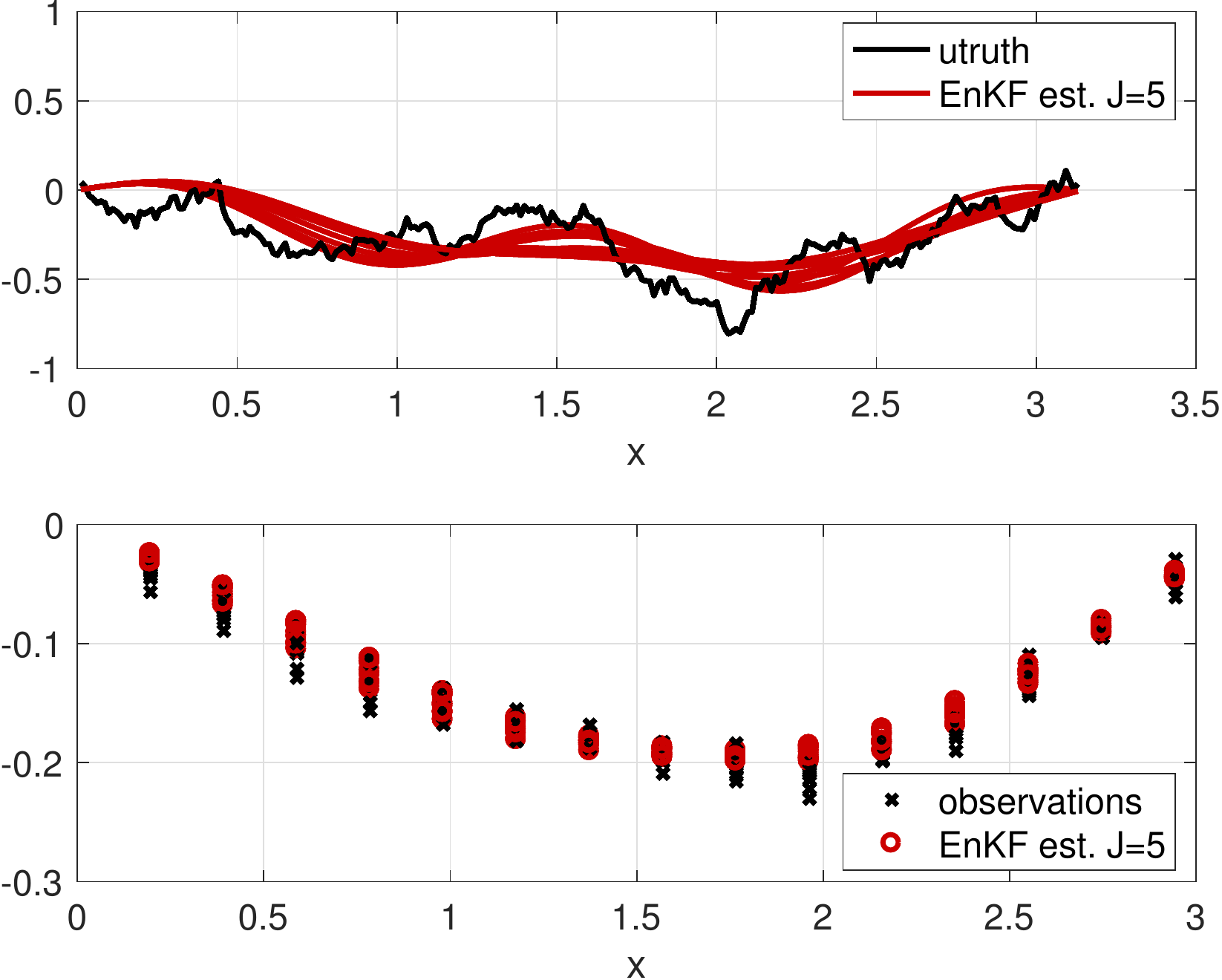}
    ~\\[-0.75cm]\caption{\footnotesize \label{fig:disc5sol}
Comparison of the \cls{ensemble Kalman inversion} estimate with the truth and the observations with discrepancy stopping rule, $J=5$ based on KL expansion  of ${C_0}=10(-\Delta)^{-1}$ (red), $K=2^4-1$, $10$ randomly initialised ensembles, $10$ randomly perturbed observation.}
 \end{figure}
 \end{minipage}
~
\begin{minipage}{0.45\textwidth}
 \begin{figure}[H]
\centering
     \includegraphics[width=\textwidth]{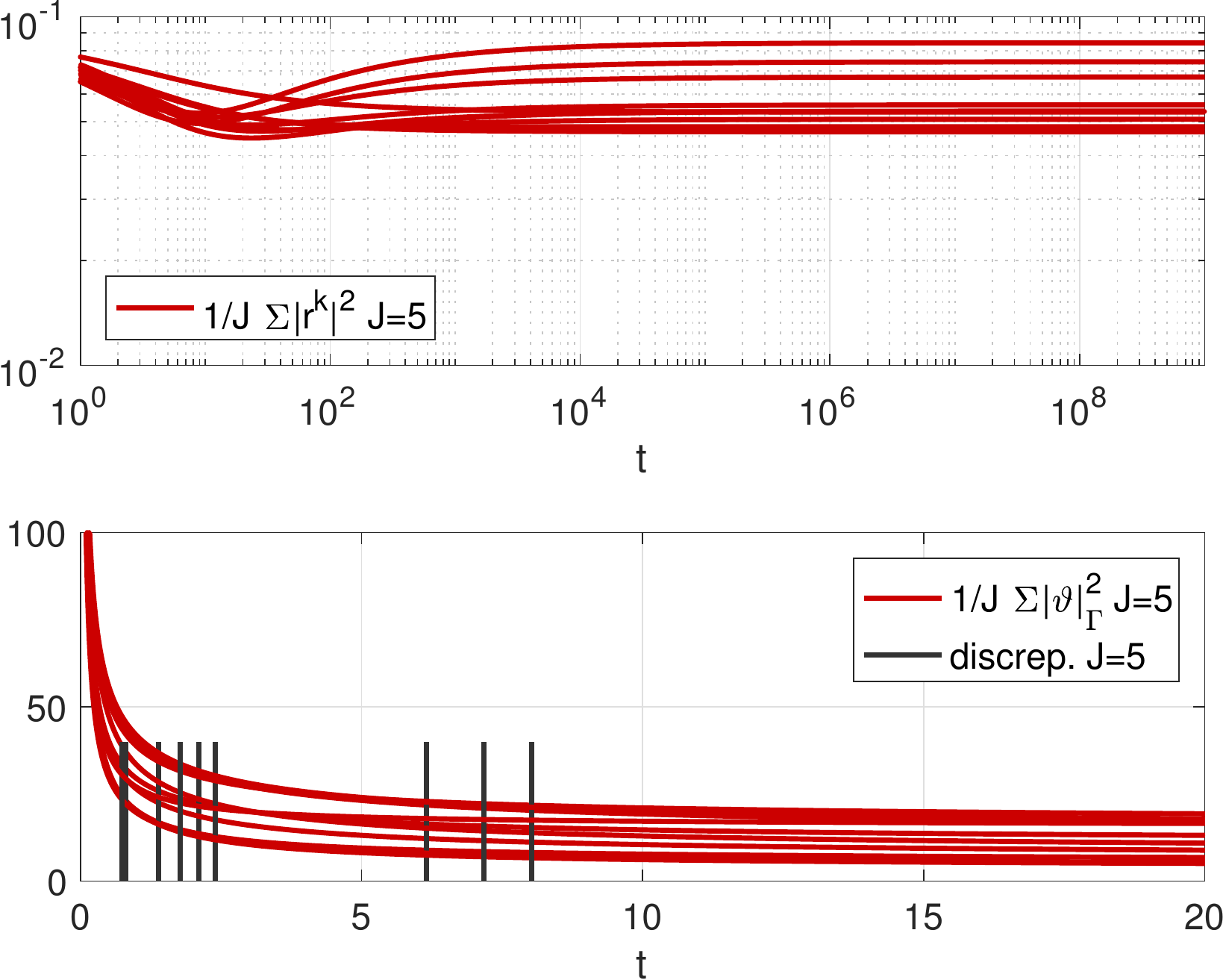}
~\\[-0.75cm]\caption{\footnotesize\label{fig:disc5}
Mapped residuals $|A\bar r|_\Gamma^2$ (above) and $|\bar \vartheta|_\Gamma^2$ with discrepancy stopping rule (below) w.r. to time $t$, $J=5$ based on KL expansion  of ${C_0}=10(-\Delta)^{-1}$ (red),  $K=2^4-1$, $10$ randomly initialised ensembles, $10$ randomly perturbed observation.}
 \end{figure}
\end{minipage}

 \begin{minipage}{0.45\textwidth}
 \begin{figure}[H]
\centering
     \includegraphics[width=\textwidth]{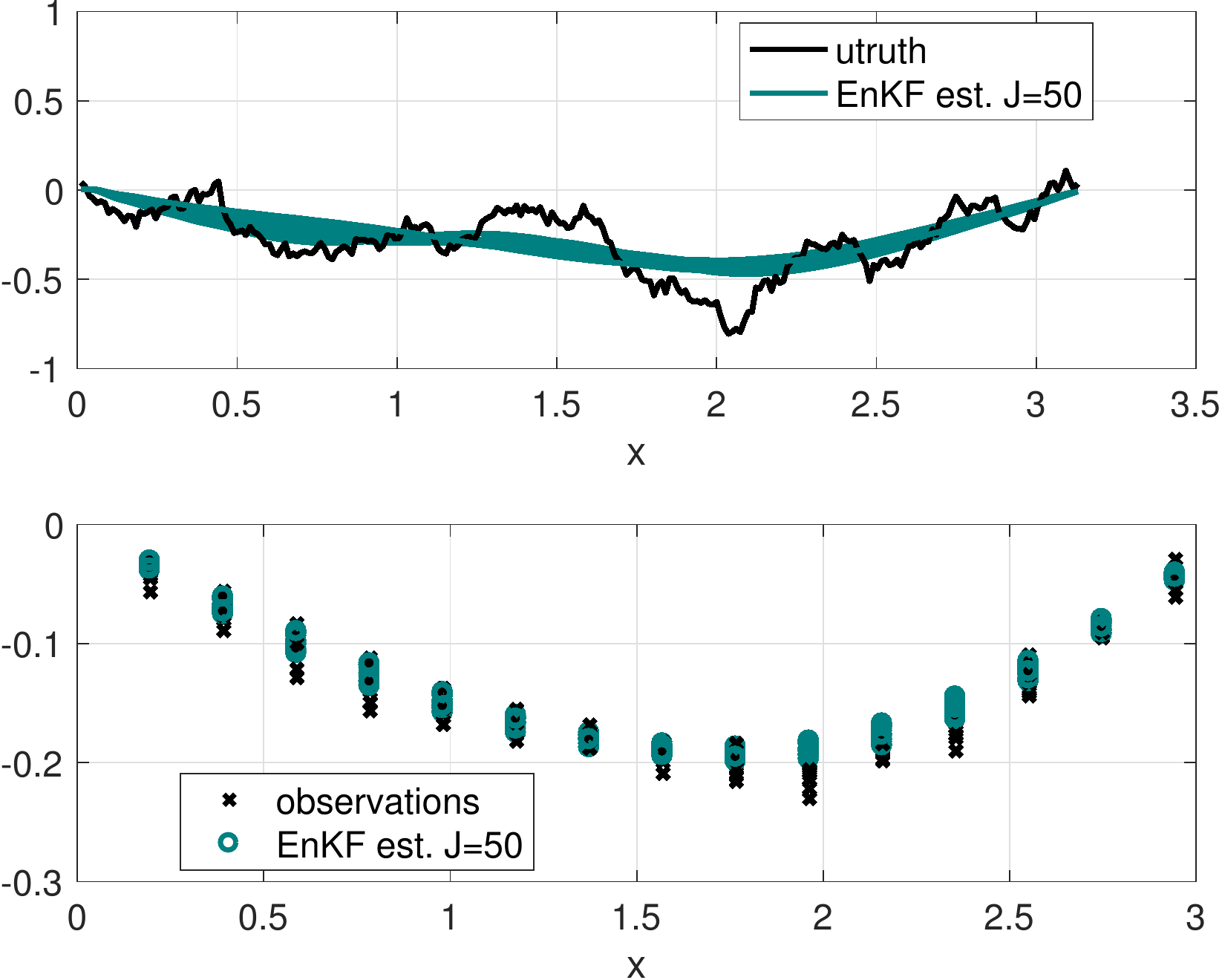}
    ~\\[-0.75cm]\caption{\footnotesize \label{fig:disc50sol}
Comparison of the \cls{ensemble Kalman inversion} estimate (with stopping rule) with the truth and the observations with discrepancy stopping rule, $J=50$ based on KL expansion  of ${C_0}=10(-\Delta)^{-1}$ (blue), $K=2^4-1$, $10$ randomly initialised ensembles, $10$ randomly perturbed observation.}
 \end{figure}
 \end{minipage}
~
\begin{minipage}{0.45\textwidth}
 \begin{figure}[H]
\centering
     \includegraphics[width=\textwidth]{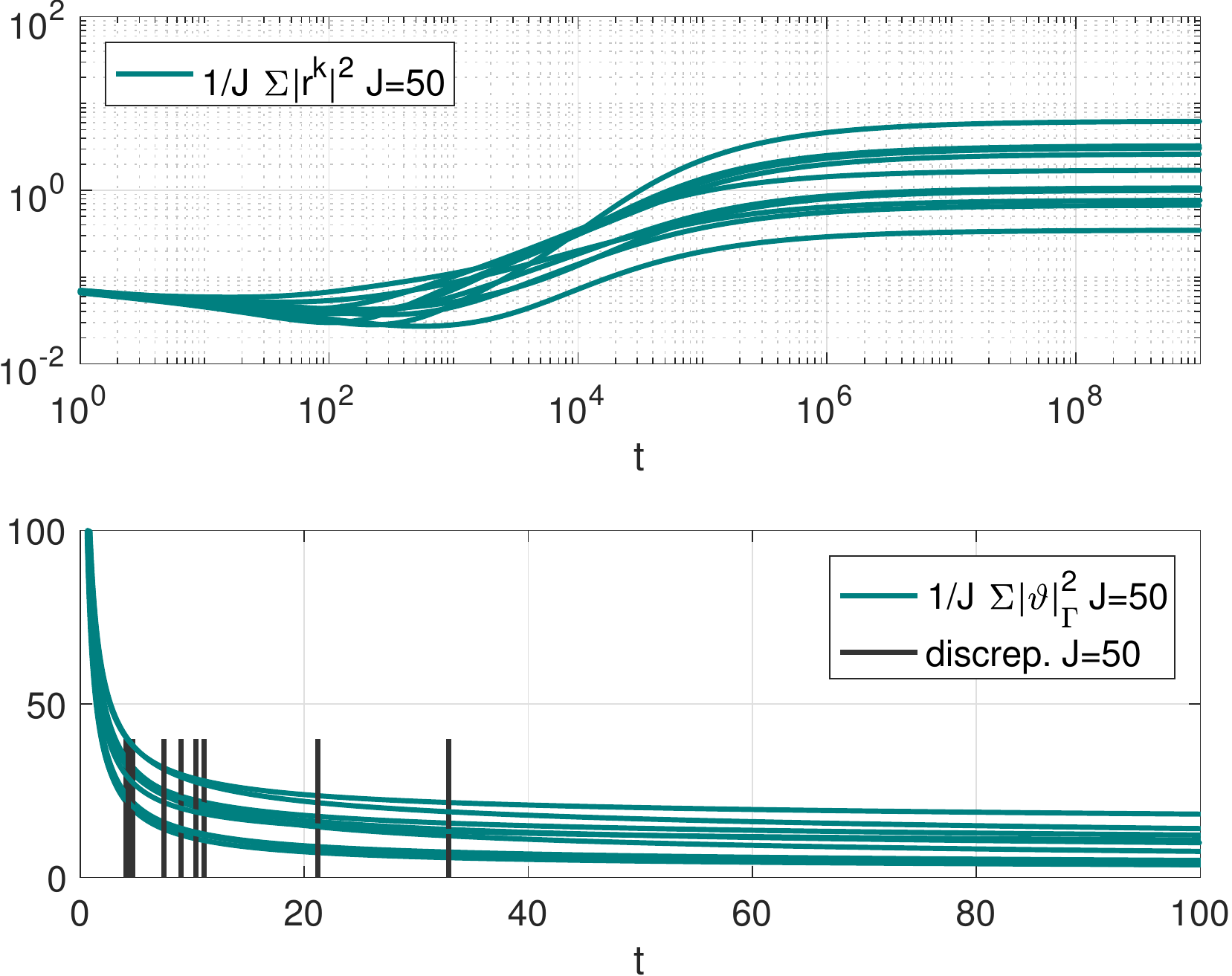}
~\\[-0.75cm]\caption{\footnotesize\label{fig:disc50}
Mapped residuals $|A\bar r|_\Gamma^2$ (above) and $|\bar \vartheta|_\Gamma^2$ with discrepancy stopping rule (below) w.r. to time $t$, $J=50$ based on KL expansion  of ${C_0}=10(-\Delta)^{-1}$ (blue),  $K=2^4-1$, $10$ randomly initialised ensembles, $10$ randomly perturbed observation.}
 \end{figure}
\end{minipage}

~\\
As the deterministic discrepancy principle is not well-defined in the high- and infinite dimensional setting, the experiments are repeated with the modified, symmetrised discrepancy principle. Results are presented in Figures \ref{fig:disc5solmod}-\ref{fig:disc50mod}.

 \begin{minipage}{0.45\textwidth}
 \begin{figure}[H]
\centering
     \includegraphics[width=\textwidth]{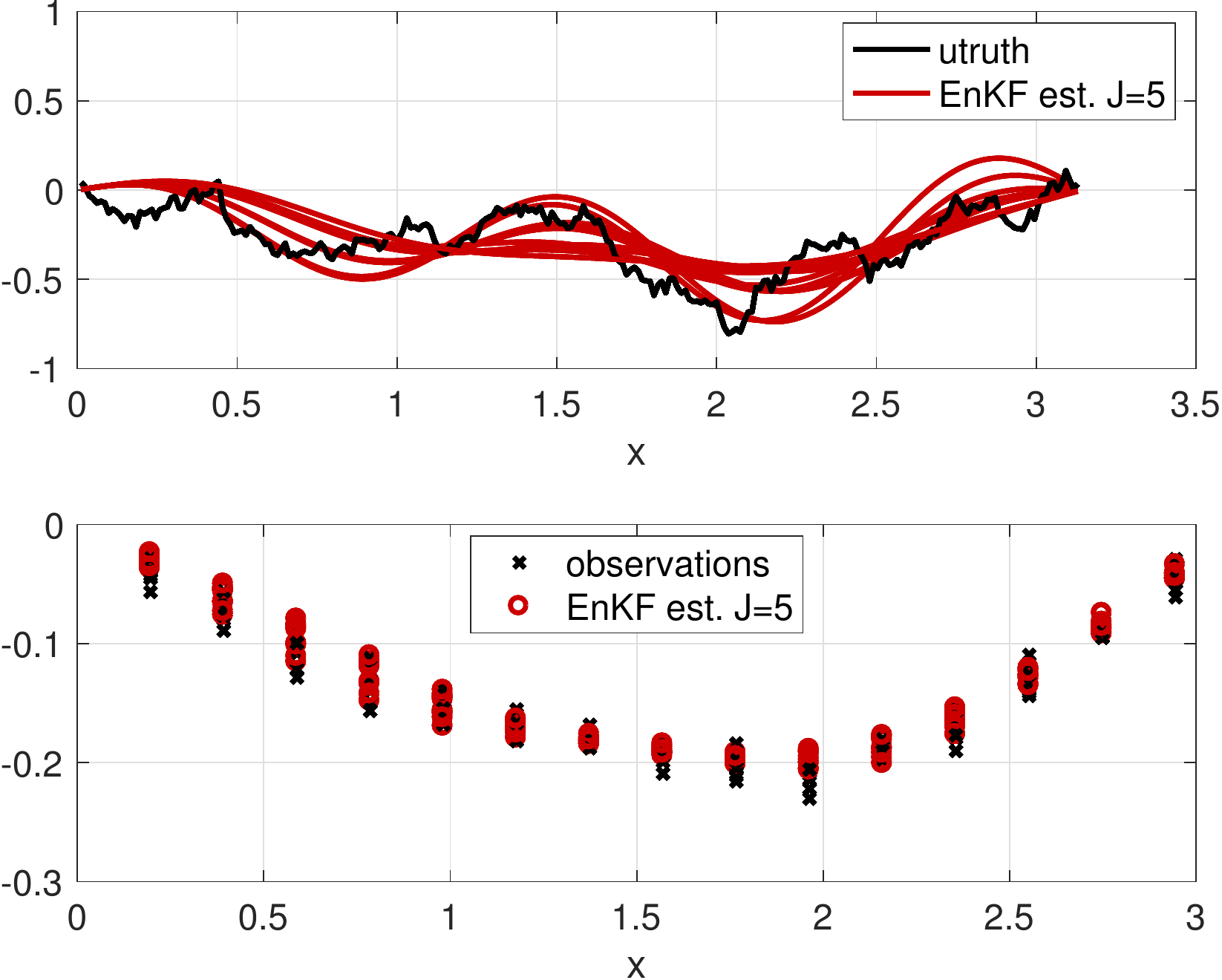}
    ~\\[-0.75cm]\caption{\footnotesize \label{fig:disc5solmod}
Comparison of the \cls{ensemble Kalman inversion} estimate with the truth and the observations with modified discrepancy stopping rule, $J=5$ based on KL expansion  of ${C_0}=10(-\Delta)^{-1}$ (red), $K=2^4-1$, $10$ randomly initialised ensembles, $10$ randomly perturbed observation, $\lambda=10^{-4}$.}
 \end{figure}
 \end{minipage}
~
\begin{minipage}{0.45\textwidth}
 \begin{figure}[H]
\centering
     \includegraphics[width=\textwidth]{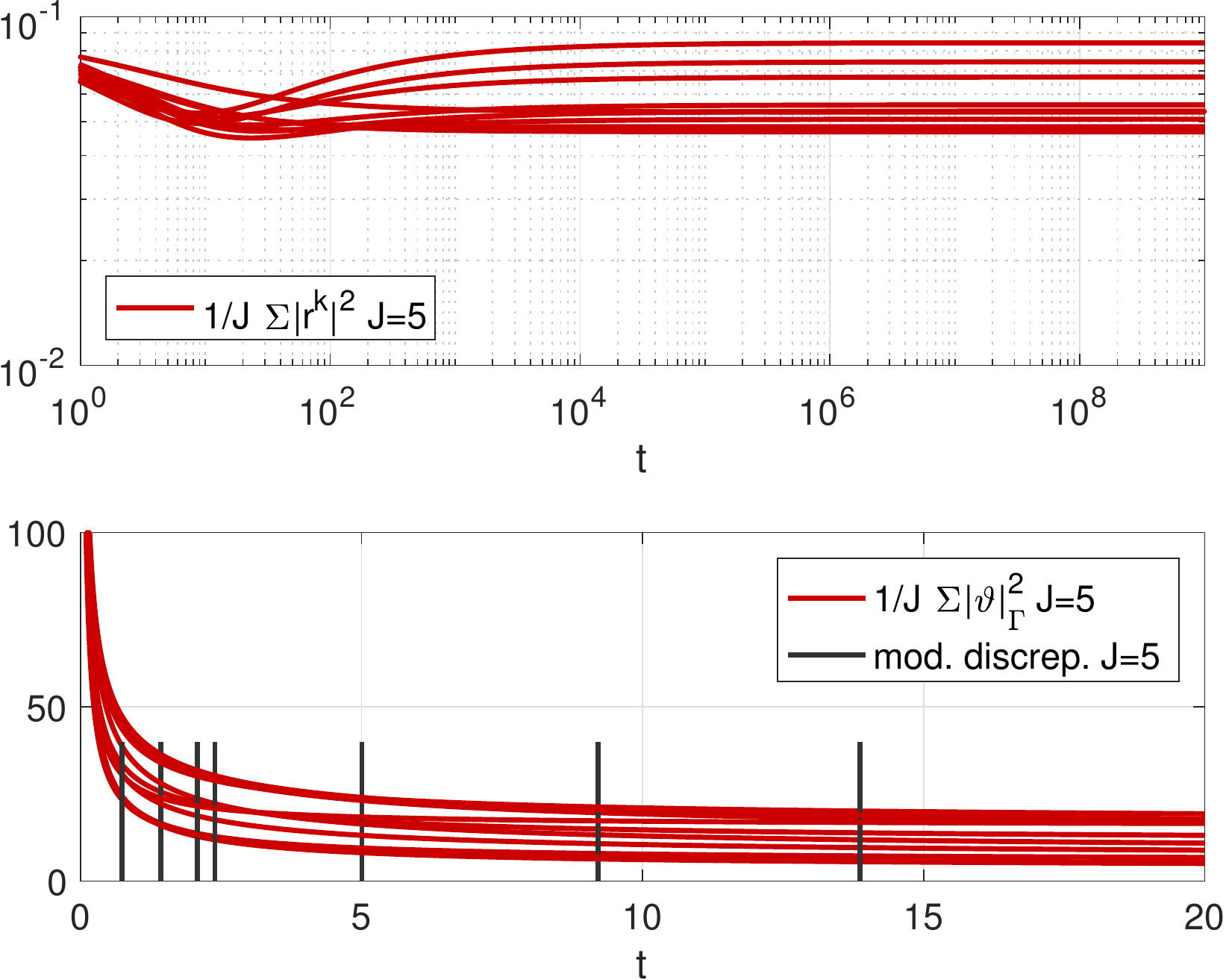}
~\\[-0.75cm]\caption{\footnotesize\label{fig:disc5mod}
Mapped residuals $|A\bar r|_\Gamma^2$ (above) and $|\bar \vartheta|_\Gamma^2$ with modified discrepancy stopping rule (below) w.r. to time $t$, $J=5$ based on KL expansion  of ${C_0}=10(-\Delta)^{-1}$ (red),  $K=2^4-1$, $10$ randomly initialised ensembles, $10$ randomly perturbed observation, $\lambda=10^{-4}$.}
 \end{figure}
\end{minipage}

 \begin{minipage}{0.45\textwidth}
 \begin{figure}[H]
\centering
     \includegraphics[width=\textwidth]{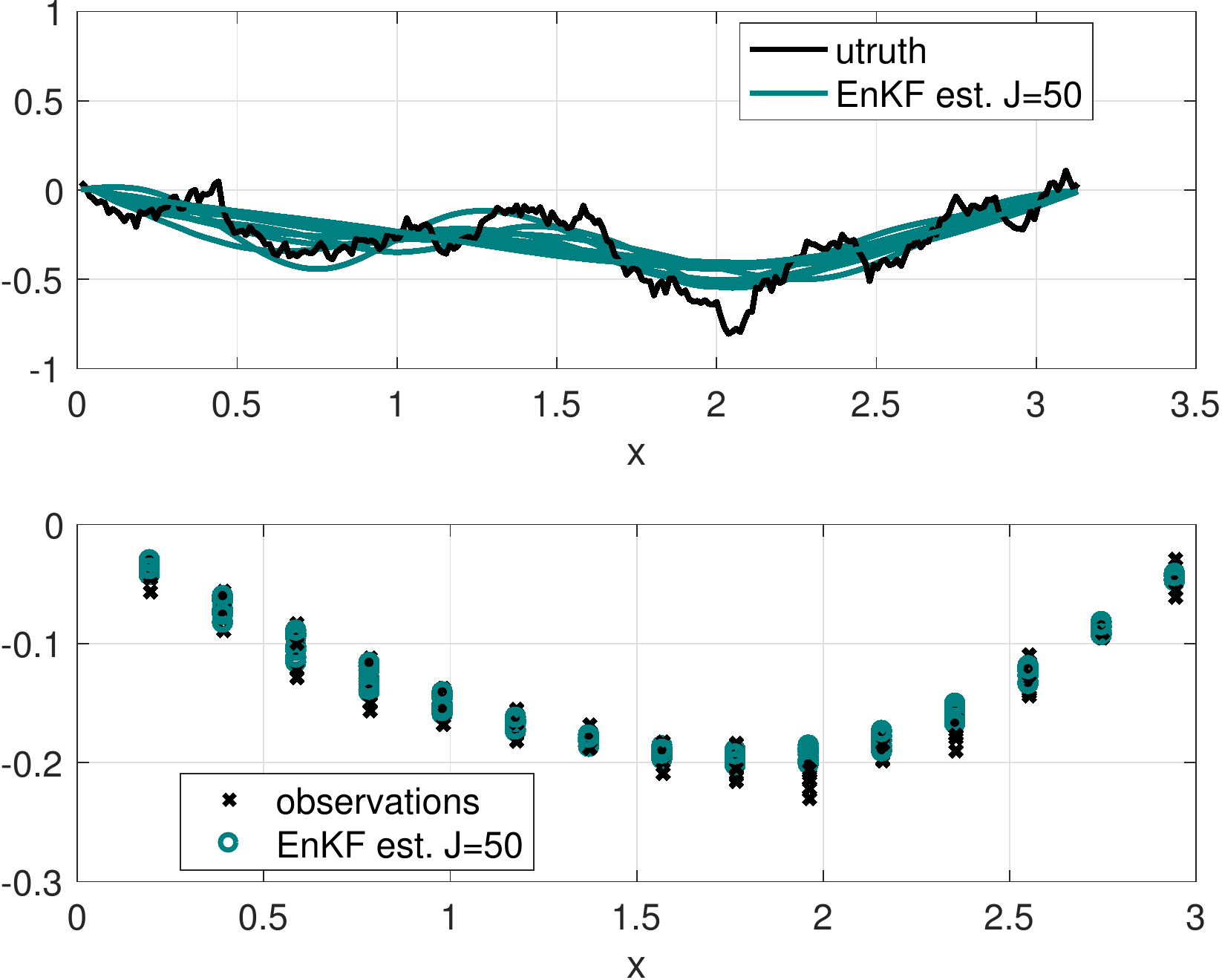}
    ~\\[-0.75cm]\caption{\footnotesize \label{fig:disc50solmod}
Comparison of the \cls{ensemble Kalman inversion} estimate (with stopping rule) with the truth and the observations with modified discrepancy stopping rule, $J=50$ based on KL expansion  of ${C_0}=10(-\Delta)^{-1}$ (blue), $K=2^4-1$, $10$ randomly initialised ensembles, $10$ randomly perturbed observation, $\lambda=10^{-4}$.}
 \end{figure}
 \end{minipage}
~
\begin{minipage}{0.45\textwidth}
 \begin{figure}[H]
\centering
     \includegraphics[width=\textwidth]{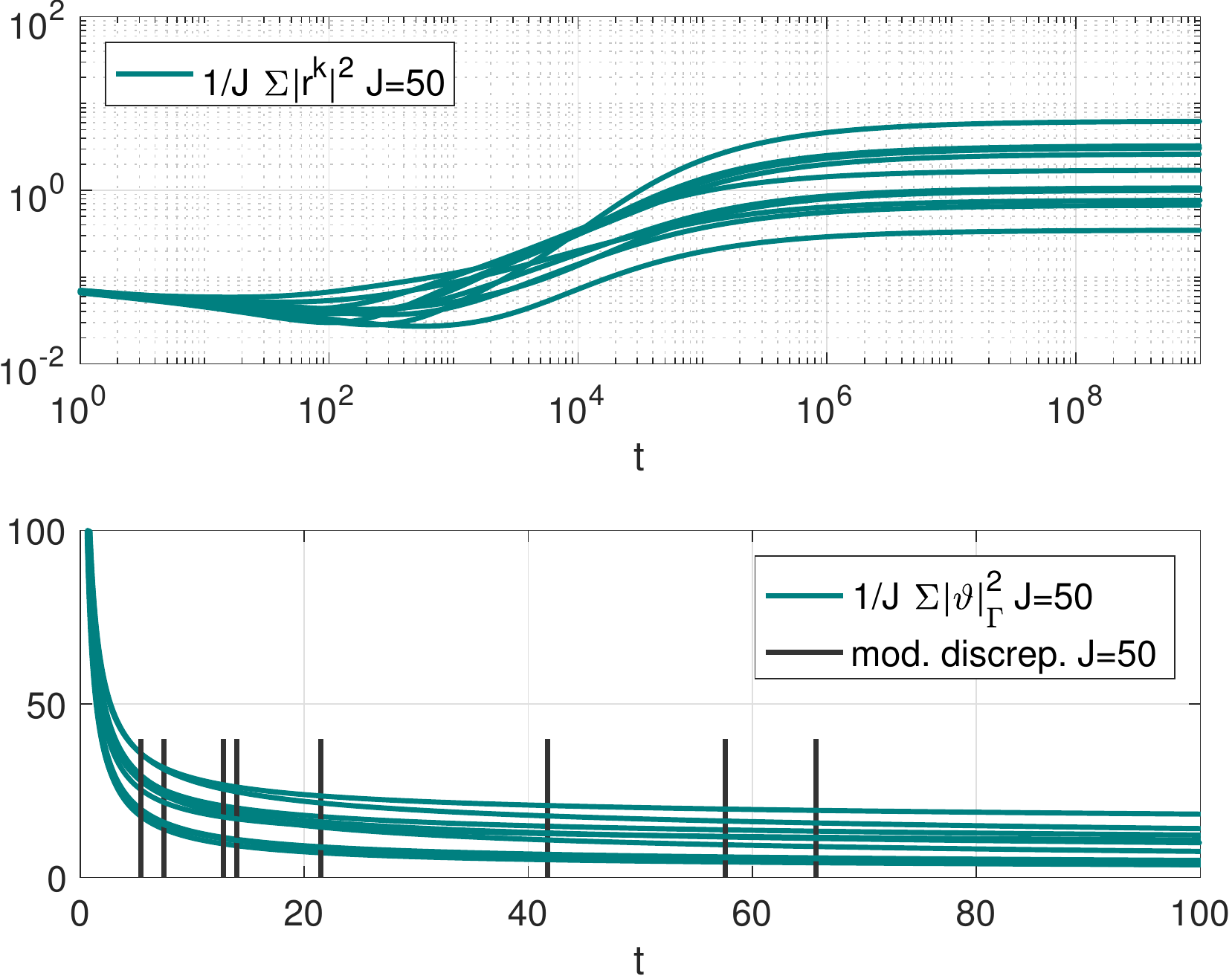}
~\\[-0.75cm]\caption{\footnotesize\label{fig:disc50mod}
Mapped residuals $|A\bar r|_\Gamma^2$ (above) and $|\bar \vartheta|_\Gamma^2$ with modified discrepancy stopping rule (below) w.r. to time $t$, $J=50$ based on KL expansion  of ${C_0}=10(-\Delta)^{-1}$ (blue),  $K=2^4-1$, $10$ randomly initialised ensembles, $10$ randomly perturbed observation, $\lambda=10^{-4}$.}
 \end{figure}
\end{minipage}

~\\
~\\
We observe that the overfitting effect is much more pronounced for the larger ensemble of size $50$, cp. the empirical residuals in Figure \ref{fig:disc5} and Figure \ref{fig:disc50}. The KL expansion of the first $50$ terms includes more fine-scale (oscillatory) details, which can be fitted to the noise in the observational data and therefore cause the overfitting effect. The smaller ensemble based on the first $5$ terms of the KL expansion avoids the overfitting effect due to the smaller ensemble size leading to a faster ensemble collapse, but also due to the smoothness of the first KL terms, i.e. the subspace property preserves the smoothness of the KL terms. Furthermore, we note that the discrepancy principle leads in all experiments to a stopping time larger than $1$ (Bayesian stopping rule), which leads for all experiments to a further improvement in the estimate. Due to the delayed ensemble collapse, the stopping times for the larger ensemble are on average greater than the ones for the smaller ensemble. The experiments suggest that an a posteriori stopping rule can significantly improve the performance of the \cls{ensemble Kalman inversion}. This observation is consistent with previous works on stopping rules for \cls{ensemble Kalman inversion}, cp. \cite{iglesias2014iterative}. 

\section{Conclusions}

The presented analysis of the ensemble Kalman filter for inverse problems shows that the well-posedness results and the quantification of the ensemble collapse derived in \cite{AndrewC2017} can be straightforwardly generalised to the noisy case. However, the convergence behaviour of the ensemble is strongly affected by the noise in the observational data and no convergence rate of the \cls{mapped }residuals can be proven: the convergence rate can be arbitrarily slow. The numerical experiments confirm the theory. In addition, the numerical experiments demonstrate the importance of an appropriate stopping rule in the presence of noise in order to avoid the well-known overfitting effect. It is also shown that the ensemble itself has a regularisation effect, caused by the ensemble collapse as well as by the chosen initialisation of the ensemble in terms of the KL expansion. \cls{Variants of the methods such as variance inflation or localisation may delay or prevent the ensemble collapse, thus they strongly influence the regularisation of the method. Stopping rules need to take this into account in order to avoid overfitting; however, the optimal strategy to use may strongly depend on the variant of
the algorithm which is used.} Even though the presented results are confined to the linear case, they provide useful insights into the performance of the filter in the presence of noise and can also enhance our understanding of the nonlinear case.  
~\\
~\\
\noindent{\bf Acknowledgments} Both authors are grateful to the careful reading,
and suggestions, of two referees. They are also grateful to the 
EPSRC Programme Grant EQUIP for funding of this research. AMS is 
also thanks DARPA (W911NF-15-2-0121) and ONR (N00014-17-1-2079)
for funding parts of this research.\\[0.55cm]

\bibliography{ref}

\begin{thebibliography}{10}

\bibitem{bergemann2010localization}
{\sc K.~Bergemann and S.~Reich}, {\em A localization technique for ensemble
  {{Kalman}} filters}, Quarterly Journal of the Royal Meteorological Society,
  136 (2010), pp.~701--707.

\bibitem{bergemann2010mollified}
\leavevmode\vrule height 2pt depth -1.6pt width 23pt, {\em A mollified ensemble
  {{Kalman}} filter}, Quarterly Journal of the Royal Meteorological Society,
  136 (2010), pp.~1636--1643.

\bibitem{Blanchard}
{\sc G.~Blanchard and P.~Math\'e}, {\em Discrepancy principle for statistical
  inverse problems with application to conjugate gradient iteration}, Inverse
  Problems, 28 (2012), p.~115011.

\bibitem{DashtiStuart}
{\sc M.~Dashti and A.M.Stuart}, {\em The {B}ayesian approach to inverse
  problems}, in Handbook of Uncertainty Quantification, R.~Ghanem, D.~Higdon,
  and H.~Owhadi, eds., Springer, 2015.

\bibitem{2016arXiv161206065D}
{\sc J.~{de Wiljes}, S.~{Reich}, and W.~{Stannat}}, {\em {Long-time stability
  and accuracy of the ensemble {Kalman}-Bucy filter for fully observed
  processes and small measurement noise}}, ArXiv e-prints,  (2016).

\bibitem{engl1996regularization}
{\sc H.~Engl, M.~Hanke, and A.~Neubauer}, {\em Regularization of Inverse
  Problems}, vol.~375, Springer Science \& Business Media, 1996.

\bibitem{ernst2015analysis}
{\sc O.~Ernst, B.~Sprungk, and H.~Starkloff}, {\em Analysis of the ensemble and
  polynomial chaos {K}alman filters in {B}ayesian inverse problems}, arXiv
  preprint arXiv:1504.03529,  (2015).

\bibitem{evensen2003ensemble}
{\sc G.~Evensen}, {\em The ensemble {K}alman filter: Theoretical formulation
  and practical implementation}, Ocean dynamics, 53 (2003), pp.~343--367.

\bibitem{Evensen06}
\leavevmode\vrule height 2pt depth -1.6pt width 23pt, {\em Data Assimilation:
  The Ensemble {Kalman} Filter}, Springer-Verlag New York, Inc., Secaucus, NJ,
  USA, 2006.

\bibitem{evensen1996assimilation}
{\sc G.~Evensen and P.~Van~Leeuwen}, {\em Assimilation of geosat altimeter data
  for the agulhas current using the ensemble kalman filter with a
  quasi-geostrophic model}, Monthly Weather, 128 (1996), pp.~85--96.

\bibitem{Ghil81}
{\sc M.~Ghil, J.~Cohn, S.E. anf~Tavantzis, K.~Bube, and I.~E.}, {\em
  Application of estimation theory to numerical weather prediction}, in Dynamic
  Meteorology: Data Assimilation Methods, Springer, 1981, pp.~139--224.

\bibitem{gratton2014convergence}
{\sc S.~Gratton, J.~Mandel, et~al.}, {\em On the convergence of a non-linear
  ensemble {K}alman smoother}, arXiv preprint arXiv:1411.4608,  (2014).

\bibitem{houtekamer2001sequential}
{\sc P.~Houtekamer and H.~Mitchell}, {\em A sequential ensemble {K}alman filter
  for atmospheric data assimilation}, Monthly Weather Review, 129 (2001),
  pp.~123--137.

\bibitem{iglesias2014iterative}
{\sc M.~Iglesias}, {\em Iterative regularization for ensemble data assimilation
  in reservoir models}, Computational Geosciences,  (2014), pp.~1--36.

\bibitem{iglesias2013ensemble}
{\sc M.~Iglesias, K.~Law, and A.~Stuart}, {\em Ensemble {K}alman methods for
  inverse problems}, Inverse Problems, 29 (2013), p.~045001.

\bibitem{iglesias2015regularizing}
{\sc M.~A. Iglesias}, {\em A regularizing iterative ensemble {{Kalman}} method
  for pde-constrained inverse problems}, arXiv preprint arXiv:1505.03876,
  (2015).

\bibitem{Kaipio}
{\sc J.~Kaipio and E.~Somersalo}, {\em Statistical inverse problems:
  Discretization, model reduction and inverse crimes}, J. Comput. Appl. Math.,
  198 (2007), pp.~493--504.

\bibitem{kelly2014well}
{\sc D.~Kelly, K.~Law, and A.~Stuart}, {\em Well-posedness and accuracy of the
  ensemble {K}alman filter in discrete and continuous time}, Nonlinearity, 27
  (2014), p.~2579.

\bibitem{Kelly25082015}
{\sc D.~Kelly, A.~J. Majda, and X.~T. Tong}, {\em Concrete ensemble {K}alman
  filters with rigorous catastrophic filter divergence}, Proceedings of the
  National Academy of Sciences, 112 (2015), pp.~10589--10594.

\bibitem{kwiatkowski2015convergence}
{\sc E.~Kwiatkowski and J.~Mandel}, {\em Convergence of the square root
  ensemble {K}alman filter in the large ensemble limit}, SIAM/ASA Journal on
  Uncertainty Quantification, 3 (2015), pp.~1--17.

\bibitem{KLS}
{\sc K.~Law, A.~Stuart, and K.~Zygalakis}, {\em Data Assimilation: A
  Mathematical Introduction}, Springer, 2015.

\bibitem{li2007iterative}
{\sc G.~Li and A.~Reynolds}, {\em An iterative ensemble kalman filter for data
  assimilation}, in SPE Annual Technical Conference and Exhibition, 2007.

\bibitem{oliver2008inverse}
{\sc D.~Oliver, A.~Reynolds, and N.~Liu}, {\em Inverse theory for petroleum
  reservoir characterization and history matching}, Cambridge University Press,
  2008.

\bibitem{reich2011}
{\sc S.~Reich}, {\em A dynamical systems framework for intermittent data
  assimilation}, BIT Numerical Mathematics, 51 (2011), pp.~235--249.

\bibitem{reich2015probabilistic}
{\sc S.~Reich and C.~Cotter}, {\em Probabilistic Forecasting and Bayesian Data
  Assimilation}, Cambridge University Press, 2015.

\bibitem{AndrewC2017}
{\sc C.~Schillings and A.~Stuart}, {\em {Analysis of the ensemble {K}alman
  filter for inverse problems}}, SIAM Numerical Analysis (accepted),  (2017).

\bibitem{Stuart_2010}
{\sc A.~M. Stuart}, {\em Inverse problems: A {B}ayesian perspective}, Acta
  Numerica, 19 (2010), p.~451Ð559.

\bibitem{tonginfl}
{\sc X.~T. Tong, A.~J. Majda, and D.~Kelly}, {\em Nonlinear stability of the
  ensemble {Kalman} filter with adaptive covariance inflation},
  arXiv:1507.08319,  (2015).

\bibitem{tongnonlinear}
\leavevmode\vrule height 2pt depth -1.6pt width 23pt, {\em Nonlinear stability
  and ergodicity of ensemble based {Kalman} filters}, Nonlinearity, 29 (2016),
  p.~657.

\end{thebibliography}
\bibliographystyle{siam}
\end{document}